\documentclass{article}

\usepackage[utf8]{inputenc}
\usepackage[T1]{fontenc}
\usepackage{authblk}
\usepackage{amsmath}
\usepackage{amssymb}
\usepackage{amsthm}
\usepackage{commath}
\usepackage{graphicx}
\usepackage{caption}
\usepackage{subcaption}
\usepackage{float}
\usepackage{algorithm}
\usepackage{algpseudocode}
\usepackage{makecell}
\usepackage[colorlinks=true, allcolors=black]{hyperref}
\usepackage[backend=biber,maxnames=3,url=false,
  style=numeric-comp,isbn=false,giveninits=true,eprint=false,
  sorting=none,doi=false,date=year]{biblatex}
\bibliography{paper_arxiv.bbl}

\DeclareMathOperator{\Rb}{\mathbb{R}}
\DeclareMathOperator{\Cb}{\mathbb{C}}
\DeclareMathOperator{\Ic}{\mathcal{I}}

\DeclareMathOperator*{\argmin}{argmin}
\DeclareMathOperator*{\argmax}{argmax}

\newtheorem{theorem}{Theorem}[section]
\newtheorem{remark}[theorem]{Remark}%[section]
\newtheorem{proposition}[theorem]{Proposition}
\newtheorem{assumptions}[theorem]{Assumptions}%[section]
\newtheorem{definition}[theorem]{Definition}

\title{\vspace{-2cm}Efficient high-resolution refinement in cryo-EM with stochastic gradient descent}

\author[1,2]{Bogdan Toader}
\author[3]{Marcus A. Brubaker}
\author[2]{Roy R. Lederman}

\affil[1]{MRC Laboratory of Molecular Biology}
\affil[2]{Yale University}
\affil[3]{York University}

\begin{document} 
\maketitle 

\begin{abstract}
Electron cryo-microscopy (cryo-EM) is an imaging technique widely used in structural biology to determine the three-dimensional structure of biological molecules from noisy two-dimensional projections with unknown orientations. As the typical pipeline involves processing large amounts of data, efficient algorithms are crucial for fast and reliable results. The stochastic gradient descent (SGD) algorithm has been used to improve the speed of ab initio reconstruction, which results in a first, low-resolution estimation of the volume representing the molecule of interest, but has yet to be applied successfully in the high-resolution regime, where expectation-maximization algorithms achieve state-of-the-art results, at a high computational cost. In this article, we investigate the conditioning of the optimization problem and show that the large condition number prevents the successful application of gradient descent-based methods at high resolution. Our results include a theoretical analysis of the condition number of the optimization problem in a simplified setting where the individual projection directions are known, an algorithm based on computing a diagonal preconditioner using Hutchinson’s diagonal estimator, and numerical experiments showing the improvement in the convergence speed when using the estimated preconditioner with SGD. The preconditioned SGD approach can potentially enable a simple and unified approach to ab initio reconstruction and high-resolution refinement with faster convergence speed and higher flexibility, and our results are a promising step in this direction.
\end{abstract}

\section{Introduction}
We consider the problem of gradient-based optimization for tomographic reconstruction with a particular focus on electron cryo-microscopy (cryo-EM).
Stochastic gradient descent (SGD) based optimization methods have become a standard algorithmic framework in many areas due to their speed, robustness, flexibility and ease of implementation, particularly with the availability of fast and robust automatic differentiation libraries.
However, in cryo-EM, SGD methods have only really found use in ab initio structure determination, where the goal is only a low-resolution structure.
High-resolution structures are then determined by switching to a different form of optimization, typically some form of expectation-maximization (EM), despite the fact that these methods are often slow and require specific modeling assumptions.
A natural question which we address here then is why stochastic gradient optimization techniques have not been able to solve the high-resolution optimization problem.
Doing so could further speed up data processing for cryo-EM, simplify workflows and unify open research questions.
However, perhaps more significantly, using SGD-based optimization methods would allow for more flexibility in modeling the reconstruction problem.
Common modeling assumptions (e.g., Gaussian noise, Gaussian priors, rigid structures,  discrete Fourier-based structure representations) which may not be optimal but are required by current refinement methods could be relaxed with SGD methods.

In this paper, we show that standard SGD-based methods struggle to accurately determine high-resolution information in cryo-EM due to fundamental properties of the induced optimization problem.
We perform a theoretical analysis, in a simplified setting, which shows that the induced optimization problem can suffer from ill-conditioning, which results in arbitrarily slow convergence for standard SGD algorithms.
Moreover, our analysis shows that the conditioning may be acceptable at low resolution but becomes worse as resolution increases, explaining why SGD has been successful in the ab initio setting but has yet to be successfully used for high-resolution refinement.
While our analysis is in a simplified setting, we argue and empirically verify that this poor conditioning behavior continues to exist in real-world cryo-EM settings.
Finally, based on our analysis, we propose a new SGD-based algorithm which, unlike standard methods, exploits higher-order derivatives to improve the conditioning of the problem.
Our results demonstrate that this new method is able to overcome the conditioning problem and efficiently and accurately determine high-frequency information.

\subsection{Background}
Cryo-EM enables biologists to analyze the structure of macromolecules in their native state.
In comparison with the older method of X-ray crystallography, cryo-EM does not require crystallized samples, allowing one to study larger molecules with complex structure and conformations.
Indeed, its potential to uncover the structure and function of macromolecules has been acknowledged by the scientific community: cryo-EM was named the ``Method of the year'' in 2015 by Nature Methods~\cite{noauthor_method_2016}, and its development was the subject of the 2017 Nobel Prize in Chemistry.

The standard cryo-EM single particle analysis (SPA) pipeline involves freezing  a biological sample in a thin layer of ice and imaging it with an electron microscope.
The imaged sample contains multiple copies of a macromolecule captured in distinct,  random and unknown orientations.
Following a number of data processing steps, two-dimensional projections of the electrostatic potential of the macromolecule are captured in a stack of images, which we refer to as the particle images. 
The goal of the cryo-EM SPA pipeline is to reconstruct a three-dimensional volume representing the structure of a molecule from the collected particle images.

In addition to the random orientations, each projection of the volume is shifted off the center of the image by a small, unknown amount, and the particle images are further blurred by a contrast transfer function (CTF) which is image-specific and depends on the settings of the microscope.
Moreover, to avoid the damaging of the biological sample by the electron beam, the imaging is done at a low dosage, which results in a poor signal-to-noise ratio (SNR).
Therefore, cryo-EM reconstruction requires solving a tomography problem with unknown viewing directions and in-plane shifts, and low SNR.
Here we refer to the particle orientations and the in-plane shifts collectively as the pose variables.

Cryo-EM reconstruction approaches implemented in established software packages~\cite{scheres_relion_2012,punjani_cryosparc_2017} consist of two separate stages: ab initio reconstruction and high-resolution 3D refinement. 
Ab initio reconstruction provides an initial, low-resolution estimation of the volume. 
This is a non-convex problem for which many methods have been developed, but the current state-of-the-art methods are based on stochastic gradient descent algorithms which were first used in the context of cryo-EM in the cryoSPARC software~\cite{punjani_building_2017,punjani_cryosparc_2017}. 
More recently, a similar approach has been adopted by other software packages such as RELION~\cite{kimanius_new_2021}.
After the ab initio step, high-resolution 3D refinement performs further optimization of the volume and generates a high-resolution reconstruction.
This usually employs an optimization algorithm such as 
expectation-maximization~\cite{dempster_maximum_1977} to iteratively reconstruct 
the volume and a search procedure to estimate the pose variables.
The EM algorithm has become a standard approach to high-resolution refinement~\cite{tagare_fast_2008,scheres_bayesian_2012,scheres_relion_2012,punjani_cryosparc_2017} and achieves state of the art reconstructions.
However, EM-based methods are computationally expensive, generally requiring full passes through all images at each iteration of refinement and necessitating sophisticated grid and tree search algorithms to reduce the computational costs.
Further, they are restricted to several key assumptions including Gaussian noise, a Gaussian prior on structures, a rigid structure, and a discrete Fourier-based representation of the structure.

Recently, a new class of methods for cryo-EM have emerged which aim to reconstruct not just static structures but also conformational heterogeneity, where the reconstructed 
volume is subject to different kinds of deformations~\cite{toader_methods_2023,lederman_hyper-molecules_2020,zhong_cryodrgn_2021,zhong_cryodrgn2_2021,levy_amortized_2022,chen_deep_2021,chen_integrating_2023,chen_improving_2023,punjani_3dva_2021,punjani_3dflex_2023,gupta_cryogan_2021,gupta_multi-cryogan_2020,herreros_zart_2023,esteve-yague_spectral_2023,schwab_dynamight_2024,li_cryostar_2024,gilles_cryo-em_2025,woollard_instamap_2025}.
Such methods greatly expand the capabilities of cryo-EM, e.g., with time-resolved cyro-EM \cite{torino_time-resolved_2023,bhattacharjee_time_2024}.
However, existing methods usually require a high resolution structure and known 
pose variables as input, limiting their applicability.
Moreover, they cannot use the standard EM-based optimization approaches, often using 
neural networks trained using SGD instead, and generally fail to match existing 
rigid structure refinement approaches in resolution and requiring substantially 
more computation. While motivated by the promise of capturing conformational heterogeneity, 
we focus here on the static reconstruction case.
Our results suggest that improving the performance of SGD-based methods may be 
the key to unlocking this new capability of cryo-EM.

\subsection{Comparison to prior work}
The Bayesian formulation of the cryo-EM reconstruction problem and its solution via the EM algorithm has a long history, with early works including \cite{sigworth_maximum-likelihood_1998,tagare_fast_2008,tagare_adaptive_2010} as well as their implementation for high-resolution 3D refinement in state-of-the-art software such as RELION~\cite{scheres_relion_2012,scheres_bayesian_2012} and cryoSPARC~\cite{punjani_cryosparc_2017}.
While early ab initio reconstruction methods involved heuristic approaches such as using a low-pass filtered known reconstruction of a similar structure to the one of interest, mathematically rigorous approaches based on the method of common lines have been developed in \cite{singer_three-dimensional_2011,shkolnisky_viewing_2012,greenberg_common_2017}.
The SGD algorithm introduced in a cryo-EM context in \cite{punjani_building_2017,punjani_cryosparc_2017} showed improved robustness and speed in obtaining ab initio reconstructions from scratch.
More recently, the VDAM algorithm, a gradient-based algorithm with adaptive learning
rate has been introduced in the latest version of the RELION software~\cite{kimanius_new_2021}.
This brief list of cryo-EM reconstruction algorithms is non-exhaustive and we refer the reader to more comprehensive review articles such as~\cite{singer_computational_2020,bendory_single-particle_2020}.

The aforementioned articles view the SGD algorithm and its variants as tools for the ab initio step, while the best results for high-resolution refinement are achieved using the EM algorithm.
In this work, we present an analysis of the conditioning of the reconstruction optimization problem and propose a method to improve the convergence of SGD for high-resolution refinement by on-the-fly estimation of a suitable preconditioner.
While basic preconditioners are used in cryoSPARC~\cite{punjani_cryosparc_2017} and several strategies for adapting the step size based on second order information are implemented in the VDAM algorithm~\cite{kimanius_new_2021}, neither work addresses the conditioning of the problem explicitly, and the preconditioners used are not suitable in the high-resolution regime.
In contrast, while we theoretically analyze the reconstruction problem in a simplified setting, our proposed solution is specifically designed to overcome the conditioning issue in a big data, high-resolution setting. 
We leverage ideas from the machine learning literature~\cite{vaswani_adaptive_2021,jahani_doubly_2022} to estimate a preconditioner that poses no significant additional computational cost over the ordinary SGD algorithm, does not require a particular initialization or warm start, and is straightforward to incorporate into existing SGD implementations in cryo-EM frameworks.
Finally, we provide numerical experiments that show the feasibility of our preconditioning approach for estimating high resolution information. 

\subsection{Outline}

The remaining parts of the article are structured as follows. 
In Section~\ref{sec:preliminaries}, we describe the mathematical setting of the 
cryo-EM reconstruction problem, as well as current approaches for high-resolution 
refinement using the EM algorithm and ab initio reconstruction using the
SGD algorithm. Section~\ref{sec:main} presents the main contributions. 
In Section~\ref{sec:cond}, we analyze the condition number of the linear inverse
problem in the simplified setting where the pose variables are known, while
in Sections~\ref{sec:hutch}-\ref{sec:algorithm} we describe several ideas 
that are part of a proposed construction of a preconditioner that allows SGD
to overcome the conditioning issue. In Section~\ref{sec:experiments}, we provide 
numerical experiments that validate the theoretical contributions, and in 
Section~\ref{sec:conclusion} we conclude with a summary of the main advantages 
of the proposed approach and motivate a potential extension of this work to the
fully general setting.

\section{Preliminaries}
\label{sec:preliminaries}

\subsection{Forward model}

The objective of cryo-EM reconstruction is to estimate a three-dimensional volume 
representing the shape of a molecule $v$ from a stack of particle images 
$\{x_i\}_{i=1}^N$. 
A simple and frequently used model of the image formation process consists of the
following steps: each particle image $x_i$ is formed by rotating the volume $v$ 
by a three-dimensional rotation operator $R_i$, projecting it along the $z$-axis, 
applying a small in-plane shift $T_i$, convolving the result with a contrast
transfer function (CTF) $C_i$, and adding Gaussian noise.

This model is often formulated in the Fourier domain to speed up the computation
of the projections by taking advantage of the Fourier slice theorem and the fast 
Fourier transform (FFT). The Fourier slice theorem states that the two-dimensional 
Fourier transform of a projection of a three-dimensional volume is the intersection 
of the three-dimensional Fourier transform of the volume with a plane passing through 
the origin of the coordinate axes, where the projection direction is determined 
by the orientation of the plane.

Let $v \in \Cb^{M_v}$ be the (discretized) three-dimensional volume and 
$x_i \in \Cb^{M_x}$, for $i=1,\ldots,N$, the particle images. Here, $M_v$ is the 
total number of voxels in a $M \times M \times M$ grid discretization of the volume 
and, similarly, $M_x$ is the total number of pixels in a $M \times M$ grid discretization 
of the particle images\footnote{
   Throughout this article, we will use the words ``voxel'' and ``pixel'' to refer
   to the elements of a discretized volume and image respectively, in the
   Fourier domain.
}. Without loss of generality, we assume that in this 
vectorized representation of the volume, the first $M_x = M \times M$
elements of the volume correspond to the volume slice at $z=0$ in the Fourier domain.
Moreover, note that in this representation, both the CTF $C_i$ and the in-plane
shift operator $T_i$ are diagonal matrices in $\Cb^{M_x \times M_x}$, since in 
the Fourier domain, the convolution corresponds to element-wise multiplication 
and the in-plane shift corresponds to a phase factor.
Since the volume and images are defined on discrete grids, 
the rotated volume and the initial grid coordinates are no longer aligned. 
Specifically, the value of a volume $v$ acted on by a rotation operator $R$ at
the grid coordinates $r$ is given by evaluating the volume $v$ at the rotated
coordinates $R^T r$: $(R v)(r) = v(R^T r)$. However, due to the possible misalignment
between the coordinate grid that $v$ is defined on and the rotated coordinates 
$R^T r$, the value of $v$ at $R^T r$ must be interpolated using its values at the
neighboring grid points. 

We define two projection operators that use nearest-neighbor
and trilinear interpolation, respectively. In short, projection by nearest-neighbor 
interpolation assigns to the rotated grid point the value of the volume at the 
grid point closest to the rotated grid point, while projection by trilinear 
interpolation performs linear interpolation using the closest eight grid points 
to the rotated grid point and assigns the resulting value to it, as follows:
\begin{definition}[Projection operator $P_{\phi}$]
    \label{def:proj}
    Let $\phi_i \in SO(3) \times \Cb^2$ denote the pose variable encapsulating the
    rotation matrix $ R \in \Rb^{3 \times 3} $ 
    and the diagonal shift matrix $T \in \Cb^{M_x \times M_x}$, 
    $\{r_j^{2D}\}_{j=1}^{M_x}$ the coordinates of the Fourier grid at the $z=0$ plane,
    and $\{r_k^{3D}\}_{k=1}^{M_v}$ the coordinates of the 3D Fourier grid that the volume $v$ is defined on.
    We define the following projection operators $P_{\phi} \in \Cb^{M_x \times M_v}$:
    \begin{enumerate}
        \item \textbf{Nearest-neighbor interpolation projection operator $P^{nn}_{\phi}$}:
            \begin{equation} 
                (P^{nn}_{\phi} v)_j := T_{jj} v_{k^*(j)},
                \qquad \text{for all} \quad j = 1,\ldots, M_x,
                \label{eq:def nn interp}
            \end{equation}
            where $k^*(j) = \argmin_{k \in \{1,\ldots,M_v\}} \|r_k^{3D} - R^T r_j^{2D}\|_2 $,
            In this case, the matrix $P^{nn}_{\phi}$ has exactly one non-zero 
            element in each row, and therefore 
            ${(P^{nn}_{\phi})}^* P^{nn}_{\phi} \in \Rb^{M_v \times M_v}$ 
            is diagonal.
        \item \textbf{Trilinear interpolation projection operator $P^{tri}_{\phi}$}
            \begin{equation}
                % (P^{tri}_{\phi} v)_j :=  T_{jj} v_j^*,
                (P^{tri}_{\phi} v)_j :=  T_{jj} \sum_{i=1}^8  c_{k_i(j)} v_{k_i(j)},
                \qquad \text{for all} \quad j=1,\ldots,M_x. 
                \label{eq:def tri interp}
            \end{equation}
            % where $v_j^*$ is obtained by using trilinear interpolation of $v$ at $R^T r_j^{2D}$.
            where $ k_1(j), \ldots, k_8(j) $ are the indices of the eight closest
            3D grid points to the rotated 2D grid point $R^T r_j^{2D}$ and
            $ c_{k_1(j)}, \ldots, c_{k_8(j)} $ 
            are the weights of the trilinear interpolation of the volume $v$
            at $R^T r_j^{2D}$ on these eight grid points.
    \end{enumerate}
\end{definition}
As the nearest-neighbor interpolation operator is a diagonal matrix, we will use 
it to establish theoretical results regarding the conditioning of the reconstruction
problem. However, the trilinear interpolation operator is more common in practice, 
and therefore we will show in numerical experiments that our preconditioner estimation 
method applies to this case as well.

\begin{remark}
    Since the nearest-neighbor projection operator $P^{nn}_{\phi}$ corresponds 
    to a plane slice through the volume passing through the center of the 
    coordinate axes that is approximated by the closest grid points to the plane,
    it follows from~\eqref{eq:def nn interp} that the indices of the non-zero 
    elements in the diagonal matrix $(P^{nn}_{\phi})^* P^{nn}_{\phi} \in \Rb^{M_v \times M_v}$ 
    are the indices of these nearest-neighbor elements in the vectorized 
    representation of the volume $v \in \Cb^{M_v}$.
    Moreover, the non-zero (diagonal) elements of $(P^{nn}_{\phi})^* P^{nn}_{\phi}$ are real 
    and positive, as the shift matrix $T$ is complex diagonal with elements of 
    unit absolute value.
    \label{rem:diag nn}
\end{remark}

\begin{remark}
    An alternative approach to the interpolation given in Definition~\ref{def:proj}
    is to sample the volume on the rotated grid using non-uniform 
    FFT~\cite{dutt_fast_1993, greengard_accelerating_2004,barnett_parallel_2019} 
    as done, for example, in~\cite{wang_fourier-based_2013}. In~\cite{anden_structural_2018},
    the structure of the matrix involving a projection and a backprojection is
    leveraged to obtain a fast preconditioner based on a circular convolution~\cite{tyrtyshnikov_optimal_1992}.
    However, this setup is less suitable for our problem where the goal is to use stochastic gradient descent to solve the general volume reconstruction problem with unknown poses. 
\end{remark}

Given the projection operators defined above, we can state the forward model as:
\begin{equation}
    x_i = C_i P_{\phi_i}v + \eta_i,
    \qquad i=1,\ldots,N,
    \label{eq:fwd model}
\end{equation}
where $P_{\phi_i} $ is one of the two projection operators in Definition~\ref{def:proj},
$C_i \in \Cb^{M_x \times M_x}$ is a diagonal matrix corresponding to the CTF,
and $\eta_i$ is the noise vector in the $i$-th image, with elements complex
normally distributed with variance $\sigma^2$. 
Both the rotation $R_i$ and the shift $T_i$ are specific to each image $x_i$ and 
not known. The CTF model is the same across all images, with image-specific parameters 
(e.g., defocus). We assume that the CTF model is known and its parameters have been
estimated in advance, and that the noise variance $\sigma^2$ is constant across 
particle images and pixels and has also been estimated in advance.

\subsection{Bayesian formulation and the EM algorithm}

The standard approach to high-resolution 3D refinement in cryo-EM is to solve a Bayesian 
formulation of the volume reconstruction problem with marginalization over the 
pose variables. This is solved using the expectation-maximization 
algorithm~\cite{dempster_maximum_1977}, which has first been used in 
the context of cryo-EM for aligning and denoising particle images 
in~\cite{sigworth_maximum-likelihood_1998} and further refined 
for full volume reconstruction in the software packages RELION~\cite{scheres_relion_2012}
and cryoSPARC~\cite{punjani_cryosparc_2017}.

Computing the full posterior distribution of the volume $v$ given the particle 
images $\{x_i\}_{i=1}^N$ is computationally expensive. 
Instead, the maximum-a-posteriori (MAP) problem involves computing the volume $v$ 
that maximizes the (log) posterior:
\begin{align}
    \argmax_{v\in\Cb^{M_v}} \log p(v | x_1,\ldots,x_N) 
    &= \argmax_{v\in\Cb^{M_v}} \log p(x_1,\ldots,x_N | v) + \log p(v) \nonumber \\
    &= \argmax_{v\in\Cb^{M_v}} \sum_{i=1}^N \log p(x_i|v) + \log p(v),
    \label{eq:posterior}
\end{align}
where $p(v)$ is the prior distribution of the volume $v$ and $p(x_i | v)$ is the
likelihood probability of observing the image $x_i$ given the volume $v$:
\begin{equation}
    p(x_i | v) \propto 
    \int_{SO(3)\times\Rb^2} p(x_i | \phi_i, v) p(\phi_i) \dif \phi_i, 
    \label{eq:likelihood}
\end{equation}
where $p(\phi_i)$ is the prior distribution of the pose variable $\phi_i$ of the 
particle image $x_i$ and $p(x_i|\phi_i,v)$ is the likelihood of observing 
image $x_i$ given the volume $v$ and pose $\phi_i$:
\begin{equation*}
    p(x_i|\phi_i,v) \propto \exp\left(
        -\frac{\|x_i - C_i P_{\phi_i}v\|_2^2}{2\sigma^2}
    \right).
\end{equation*}
The volume prior $p(v)$ plays a regularization role 
in solving the reconstruction problem and, throughout this manuscript, we assume 
it is Gaussian with variance $\tau^2$ for all elements of the volume:
\begin{equation*}
    p(v) \propto e^{-\frac{\|v\|_2^2}{2\tau^2}}.
\end{equation*}

The EM algorithm finds the optimal volume $v^*$ in~\eqref{eq:posterior} by alternating 
between two steps at each iteration. The fist step computes the expected value of 
the posterior of the pose variables given the current estimate 
of the volume $v^{(k)}$ (the E step), while the second step computes a new 
iterate for the volume $v^{(k+1)}$ that maximizes this quantity (the M step).
If the prior is Gaussian, the $v^{(k+1)}$ can be computed analytically by letting
the derivatives be equal to zero. 
The resulting algorithm performs the following update at iteration $k$:
\begin{subequations}
\begin{align}
    \text{E step:} &\qquad
    \Gamma_{i}^{(k)}(\phi) = p(\phi_i = \phi | v^{(k)}, x_1, \ldots, x_N)
    =\frac{
        p(x_i | \phi, v^{(k)}) p(\phi)
    }{
        \int_{\phi_l} p(x_i | \phi_l, v^{(k)}) p(\phi_l) \dif \phi_l
    },
    \label{eq:em gamma}
    \\
    \text{M step:} &\qquad
    v^{(k+1)} = 
    \left(
        \sum_{i=1}^N \int_{\phi} \Gamma_{i}^{(k)}(\phi)
        P_{\phi}^* C_i^2 P_{\phi} \dif\phi
        + \frac{\sigma^2}{\tau^2} I_{M_v}
    \right)^{-1}
    \left(
        \sum_{i=1}^N \int_{\phi} \Gamma_{i}^{(k)}(\phi)
        P_{\phi}^* C_i x_i \dif\phi
    \right),
    \label{eq:em v}
\end{align}
\end{subequations}
where $I_{M_v}$ is the $M_v \times M_v$ identity matrix.

In more general implementations of the EM algorithm in software packages such 
as RELION and cryoSPARC, the noise variance $\sigma^2$ and the prior variance 
$\tau^2$ are not constant across image pixels and volume voxels respectively, and 
are both estimated at each iteration using the current volume iterate $v^{(k)}$ 
and weights $\Gamma_i^{(k)}$.

Since the optimization landscape is non-convex, the EM algorithm can converge to 
a local maximum different from the MAP estimator~\cite{dempster_maximum_1977}, 
so a good initialization is required.
Therefore, it is best used for high-resolution refinement, where an initial volume 
and possibly priors for the pose variables are provided by the ab initio algorithm. 
However, the E step is particularly expensive at high resolution, as it requires 
integrating over the entire space of the pose variable $\phi_i$ for each particle 
image in the dataset. 
In practice, this is performed by employing efficient gridding and search techniques.

\subsection{Stochastic gradient descent for ab initio reconstruction}

The SGD algorithm has become the preferred method
for solving the MAP problem~\eqref{eq:posterior} in the context of ab initio
reconstruction. It was first used in~\cite{punjani_cryosparc_2017} to obtain an initial
volume reliably and without requiring a good initialization, as it is the case
for EM. The key property that SGD exploits is that the objective function 
in~\eqref{eq:posterior} can be split into separate terms for each particle image,
not unlike the loss functions used in the training of deep neural 
networks. In particular, %letting $f(v) = \frac{1}{N} \log p (v | x_1, \ldots, x_N)$, 
the objective function~\eqref{eq:posterior} can be written as:
\begin{equation}
    f(v) = \frac{1}{N}\sum_{i=1}^N f_i(v),
    \label{eq:f sgd}
\end{equation}
where each term $f_i$ corresponds to a particle image:
\begin{equation*}
    f_i(v) = \log p(x_i|v) + \frac{1}{N} \log p(v)
\end{equation*}
At iteration $k$, SGD performs the following update:
\begin{equation}
    v^{(k+1)} = v^{(k)} - \eta_k \frac{\dif f_j(v^{(k)})}{\dif v^{(k)}},
    \label{eq:sgd iter}
\end{equation}
where $\eta_k$ is the step size at iteration $k$ and the index $j$ of the term 
(and particle image) used to compute the gradient is chosen uniformly at random. 

In practice, a mini-batch of particle images, rather than a single image, is used 
to compute the gradient at each iteration. The volume update step~\eqref{eq:sgd iter}
for the current mini-batch is performed after a refinement step of the pose 
distribution of each particle image in the current mini-batch given the current 
volume and data $p(\phi_i | v^{(k)}, x_1,\ldots,x_N)$, for example by 
using~\eqref{eq:em gamma} similarly to the EM algorithm. Various techniques can 
be used to speed up the computation time, for example the grid refinement 
implemented in RELION~\cite{scheres_relion_2012} or the branch-and-bound approach 
in cryoSPARC~\cite{punjani_cryosparc_2017}.

The SGD algorithm has two major advantages over EM. First, it has a lower computational 
cost per iteration, as it only uses a subset of the dataset, while EM requires 
a pass through the entire dataset. While computing the integrals in~\eqref{eq:likelihood} 
for the images in the mini-batch cannot be avoided, the gridding and searching 
approaches that EM uses to efficiently sample the space of poses are also beneficial 
in the implementation of SGD.
Second, it has been observed that the noise in the sampled gradient 
enables SGD to explore the optimization landscape more efficiently, 
preventing it from 
becoming stuck in unwanted local minima~\cite{keskar_large-batch_2017, smith_bayesian_2018}.
Given these advantages, the SGD algorithm is particularly well-suited for ab initio 
reconstruction, as it has been shown in practice in cryoSPARC~\cite{punjani_cryosparc_2017}.
More recently, the gradient-based algorithm VDAM has been introduced for ab initio
reconstruction in the RELION software~\cite{kimanius_new_2021}. However, despite 
the clear benefits in terms of computational cost and convergence speed of stochastic 
gradient algorithms compared to EM for (low resolution) ab initio reconstruction, 
EM is still the state-of-the-art approach for high-resolution refinement. 

To investigate how the performance of SGD can be improved for high-resolution
cryo-EM reconstruction, we first state a convergence result from~\cite{schmidt_fast_2013}:
\begin{theorem}[Convergence of the stochastic gradient descent algorithm]
    Suppose that $f: \Rb^M \to \Rb$ in~\eqref{eq:f sgd} is twice-differentiable
    and strongly convex and its gradient is Lipschitz-continuous with constant $L$.
    Furthermore, we assume that there exists $B \geq 1$ such that:
    \begin{equation}
        \max_i \{ \|\nabla f_i(x) \| \} \leq B \| \nabla f(x) \|.
        \label{eq:strong growth cond}
    \end{equation}
    Let $v^*$ be the minimizer of $f$ and $v^{(k)}$ the iterate generated
    by the SGD iteration~\eqref{eq:sgd iter} with fixed step size $\eta_k = \frac{1}{LB^2}$. 
    Then:
    $$
    \mathbb{E}\left[ f(v^{(k)}) \right] - f(v^*) \leq
    \left(1 - \frac{1}{\kappa(\nabla^2 f(v^*)) B^2} \right)^k
    \left[ f(v^{(0)}) - f(v^*)  \right],
    $$
    where $ \kappa(\nabla^2 f(v^*)) = \frac{\max_j \lambda_j}{\min_j \lambda_j}$
    is the condition number of the Hessian matrix $\nabla^2 f(v^*)$. 
    \label{thm:sgd convergence}
\end{theorem}
The main implication of Theorem~\ref{thm:sgd convergence} is that the rate of 
convergence of SGD is affected negatively by a large condition number of the 
Hessian matrix $\kappa(\nabla^2 f(v^*))$.
In particular, we can derive that, 
to reach objective function error 
$\epsilon = \mathbb{E}\left[ f(v^{(k)}) \right] - f(v^*)$,
at least $k^*$ iterations are required, where:
\begin{equation}
    k^* = \mathcal{O}\left(
        \kappa(\nabla^2 f(v^*)) B^2 \log\frac{1}{\epsilon}
    \right).
    \label{eq:min iters}
\end{equation}
Put informally, this means that the number of iterations required to reach a certain accuracy $\epsilon$ scales with the condition number of the Hessian or, in other words, the impact of a condition number equal to $C > 1$ is that $C$ times more iterations are required to reach the same accuracy $\epsilon$ compared to the case where the condition number is 1.
Note that we chose this particular result for its simplicity, despite the rather
strong condition~\eqref{eq:strong growth cond}. More general results for the convergence 
of SGD where this condition is relaxed can be found, for example,  
in~\cite{moulines_non-asymptotic_2011, needell_stochastic_2016, gower_sgd_2019}.

In general, one can apply a preconditioner to improve the convergence of SGD. 
Preconditioning linearly reparameterizes the optimization problem to change its condition number without changing its solution in order to improve numerical behaviour and convergence dynamics.
In particular, we will precondition the optimization problem using an approximation of the diagonal of the Hessian matrix of the loss function, which we will define precisely in the next section.
Such an approximation of the Hessian diagonal will be obtained using 
Hutchinson's diagonal estimator~\cite{bekas_estimator_2007}, which allows one to 
compute (an approximation of) the diagonal of a matrix $H$ only using 
matrix-vector products, without forming the full matrix $H$.
This is achieved by computing the expectation:
\begin{equation}
    D = \mathbb{E}[z \odot Hz],
    \label{eq:hutch}
\end{equation}
where $z \in \mathbb{R}^{M_v}$ is a vector with elements drawn from a Rademacher(0.5) 
distribution (the elements of $z$ are $1$ or $-1$ with equal probability) and 
$\odot$ denotes element-wise multiplication. 

The central claim of this article is that the condition number of the Hessian of 
the loss function in the cryo-EM reconstruction problem scales with 
the target resolution of the reconstruction, and slows the convergence 
of SGD for high resolution refinement. In the next section, we argue that this is 
indeed the case and we propose a solution for overcoming this issue in a simplified setting.

\section{SGD for high-resolution refinement: fixed pose variables}
\label{sec:main}

In this section, we study the optimization problem~\eqref{eq:posterior} in the
setting where the pose variables (the three-dimensional orientations and the in-plane 
shifts of the particles) are known. While this is a simpler problem that can be 
solved with other methods, it captures the main difficulty that makes the application 
of gradient-based algorithms non-trivial at high resolution, namely the large condition 
number of the Hessian. Therefore, analyzing the reconstruction problem with the 
known pose variables provides useful insights and directions for approaching the 
problem in its full generality.

For simplicity and without loss of generality, we assume, like in the previous section, 
that the variance of the noise $\sigma^2$ and the variance of the prior $\tau^2$ 
are fixed and given in advance, and are constant across all images and pixels 
(in the case of $\sigma^2$) and across all voxels of the volume (in the case of $\tau^2$). 
The analysis presented in this section can be generalized to the case where
$\sigma$ and $\tau$ are not constant, and the preconditioner estimation we propose 
has the flexibility to incorporate existing methods for determining
these parameters in the reconstruction process. To simplify the notation, 
we collect these two parameters in one parameter $\lambda = \frac{\sigma^2}{\tau^2}$, 
which we will refer to as the regularization parameter.

\subsection{Condition number}
\label{sec:cond}

Having access to the true pose variable $\phi_i^*$ for each image $x_i$ is 
equivalent to taking the prior distribution for the pose variable $\phi_i$ 
in~\eqref{eq:likelihood} to be a Dirac delta distribution centered at the true 
value $\phi_i^*$, namely $p(\phi_i) = \delta_{\phi_i^*}$. 
In addition, we assume for simplicity that the images are not deformed by the CTF: 
$C_i = I_{M_x}$, for all $i=1,\ldots,N$. 
With the regularization parameter $\lambda$ described above and letting  $P_i := P_{\phi_i^*}$, 
we write the optimization problem~\eqref{eq:posterior} as:
\begin{equation}
    \argmin_{v\in\Cb^{M_v}} \frac{1}{2} \sum_{i=1}^N
    \|x_i - P_i v\|_2^2 + \frac{\lambda}{2} \|v\|_2^2,
    \label{eq:optim cp}
\end{equation}
Letting $f(v)$ be the objective function in~\eqref{eq:optim cp} and $f_i(v)$
the $i$-th term: 
\begin{equation*}
    f_i(v) := \frac12 \|x_i - P_i v\|_2^2 + \frac{\lambda}{2N} \|v\|_2^2,
\end{equation*}
the optimization problem~\eqref{eq:optim cp} becomes:
\begin{equation}
    \argmin_{v\in\Cb^{M_v}} f(v) = 
    \argmin_{v\in\Cb^{M_v}} \sum_{i=1}^N f_i(v),
    \label{eq:min prob}
\end{equation}
The minimizer of~\eqref{eq:min prob} is the point $v^* \in \Cb^{M_v}$ that satisfies: 
\begin{equation}
    Hv^* = b,
    \label{eq:lin sys}
\end{equation}
where $b =  \sum_{i=1}^N P_i^* x_i$ and $H \in \mathbb{R}^{M_v \times M_v}$ with:
\begin{equation}
    H = \nabla^2 f =  \sum_{i=1}^N {P_i}^*P_i + \lambda I_{M_v}.
    \label{eq:hessian}
\end{equation}
Note that, for the problem~\eqref{eq:optim cp}, $H$ is the Hessian of the objective
function. The SGD algorithm solves problem~\eqref{eq:min prob} by iteratively 
taking steps in the direction of negative sampled gradient:
\begin{equation*}
    v^{(k+1)} = v^{(k)} - \eta_k \nabla f_j(v^{(k)}),
\end{equation*}
where $\eta_k$ is the step size at iteration $k$ and the index $j$ is selected
uniformly at random. Its convergence properties are determined by the condition 
number of the matrix $H$, as stated in Theorem~\ref{thm:sgd convergence}.

When the projection matrices $P_i$ are the nearest-neighbor interpolation 
matrices~\eqref{eq:def nn interp} in Definition~\ref{def:proj}, 
the matrices $P_i^* P_i$ are diagonal with real non-negative elements (see 
Remark~\ref{rem:diag nn}), thus the Hessian matrix $H$ is also diagonal with real 
non-negative elements. In this case, its condition number~\cite{trefethen_numerical_1997} is
\begin{equation}
    \kappa(H) = \frac{\max_i H_{ii}}{\min_i H_{ii}}.
    \label{eq:cond H}    
\end{equation}

We will now analyze the structure and the condition number of the Hessian matrix 
$H$ when the projection matrices correspond to nearest-neighbor interpolation,
namely $P_i = P_{\phi_i^*}^{nn}$, for $i=1,\ldots,N$, 
as given in Definition~\ref{def:proj}.
In order to do so, we introduce two necessary concepts. First, the projection 
assignment function of a particle image maps each element of the image to the 
element of the volume whose value is assigned to it by the projection operator.
\begin{definition}[Projection assignment function]
    Let $P_i := P_{\phi_i}^{nn} \in \Cb^{M_x \times M_v}$, $i=1,\ldots,N$, be nearest-neighbor
    interpolation projection matrices given in Definition~\ref{def:proj} and
    corresponding to the pose variables $\phi_i$, $i=1,\ldots,N$.
    We define the projection assignment function 
    $\Lambda_i : \{1,2,\ldots,M_x\} \to \{1,2,\ldots,M_v \}$ as the function that 
    maps each pixel index $k$ of the $i$-th particle image to the voxel index 
    $\Lambda_i[k]$ in the volume $v$ whose value is assigned by the operator 
    $P_i$ at index $k$. Namely, we have that:
    \begin{equation*}
        (P_i v) [k] = T_i[k] v[\Lambda_i(k)],
        \quad
        k = 1,\ldots,M_x,
    \end{equation*}
    where the square brackets notation is used for the value of the image (in the 
    left-hand side) or volume (in the right-hand side) at a particular index $k$, 
    and $T_i[k]$ is the $k$-th element in the diagonal of the translation matrix $T_i$. 
\end{definition}
Second, the voxel mapping set of a volume element contains the indices of the 
images that contain a projection of that volume element.
\begin{definition}[Voxel mapping set]
    For every voxel index $j \in \{1,\ldots,M_v\}$, we define the voxel mapping 
    set $\Omega_j \subseteq \{1,\ldots,N\}$ as the set of indices of 
    images that contain a pixel that is mapped by their projection assignment 
    functions $\Lambda_i$ to $j$, namely:
    \begin{equation*}
        \Omega_j = \{i : \exists k \in \{1,\ldots,M_x\} 
        \text{ such that } \Lambda_i(k) = j\}.
    \end{equation*}
\end{definition}
Given the functions $\Lambda_i$ and the sets $\Omega_j$ defined above, the diagonal 
elements of the (diagonal) matrix $P_i^* P_i$  are $(P_i^* P_i)_{jj} = 0$ 
if $i \notin \Omega_j$ and $(P_i^* P_i)_{jj} \in \{1, 2\}$ otherwise. 
We make the following assumptions: 
\begin{assumptions}\label{assump:projections}
\begin{enumerate}
    \item We assume that each voxel index $j$ is mapped at most once by the projection
        assignment function of an image $\Lambda_i$.
    \item Without loss of generality, we assume that $j=1$ is the index of the 
        voxel corresponding to the center of the coordinate axes. 
        Then, the voxel at $j=1$ is mapped by all projection operators $P_i$,
        $i=1,\ldots,N$, or equivalently, $\Omega_1 = N$.
\end{enumerate}
\end{assumptions}
The second assumption above concerns the ordering of the elements in the vectorized 
representation of the grid, specifically so that the center is mapped to the element 
at index $j=1$, while the first assumption simplifies our analysis (at a cost of
a factor of at most two in the condition number bound below) by ensuring that 
the diagonal elements of $P_i^* P_i$ satisfy:
\begin{equation}
    (P_i^* P_i)_{jj} = \begin{cases}
        0, &\quad\text{if}\quad i \notin \Omega_j,
        \\
        1, &\quad\text{if}\quad i \in \Omega_j.
    \end{cases}
    \label{eq:pp}
\end{equation}
Then, the full Hessian matrix $H$ is also diagonal, with its diagonal elements given by:
\begin{equation}
    H_{jj} = |\Omega_j| + \lambda,
    \label{eq:H diag}
\end{equation}
where $|\Omega_j|$ is the cardinality of the set $\Omega_j$ 
and $0 \leq |\Omega_j| \leq N$, for all $j=1,\ldots,M_v$.

Finally, Proposition~\ref{prop:cond} below captures the main difficulty of the 
reconstruction problem, namely the condition number of the Hessian matrix 
increasing with the resolution.
\begin{proposition}[Condition number bound]\label{prop:cond}
    Let $M_x = M^2$ and $M_v = M^3 $ for grid length $M$ and let Assumptions~\ref{assump:projections} 
    hold for the nearest-neighbor interpolation projection matrices $P_i := P_{\phi_i}^{nn}$.
    Then, for any fixed number of images $N$, we have that:
    \begin{subequations}
    \begin{align}
        &\kappa(H) \geq \frac{N + \lambda}{N/M + \lambda},
        \qquad \forall M \leq N,
        \label{eq:cond number square lessthan}
        \\
        &\kappa(H) = \frac{N + \lambda}{\lambda},
        \quad\qquad \forall M > N.
    \end{align}
    \label{eq:cond number square}
    \end{subequations}
\end{proposition}
\begin{proof}
    Since the matrices $P_i$ are the nearest-neighbor interpolation matrices
    $P_{\phi_i}^{nn}$, the matrix $H$ is also diagonal and, according to~\eqref{eq:cond H},
    to compute $\kappa(H)$ we need to find the largest and smallest elements of $H$. 

    Using~\eqref{eq:H diag} and Assumptions~\ref{assump:projections}, we have that
    $\max_{j} H_{jj} = H_{11} + \lambda = N+\lambda$.

    To compute $\min_{j} H_{jj}$, note that the projection assignment functions
    $\Lambda_i$ for $i=1,\ldots,N$, map $NM^2$ image pixels to $M^3$ volume voxels.

    For $M > N$, there are more voxels than total pixels (in all the
    images), and so there exists a voxel $j^*$ such that $|\Omega_{j^*}| = 0$. 
    Then, $\min_{j}H_{jj} = H_{j^* j^*} = \lambda$, 
    and so $\kappa(H) = \frac{N+\lambda}{\lambda}$.
    For $M \leq N$, there are $NM^2$ pixels mapped to $M^3$ voxels, and therefore
    there exists a voxel $j^*$ such that $|\Omega_{j^*j^*}| \leq NM^2/M^3 = N/M$.
    Then, $\min_j H_{jj} \leq N/M + \lambda$, and so
    $\kappa(H) \geq \frac{N+\lambda}{N/M + \lambda}$.
\end{proof}

We can also write the bounds in Proposition~\ref{prop:cond} in terms of the  
radius $R$ in the Fourier domain. If we assume that the number of pixels in a 2D
disk of radius $R$ is approximately $\pi R^2$ and the number of voxels in a 3D
ball of radius $R$ is approximately $\frac{4}{3} \pi R^3$ then, following a similar 
argument, we obtain:
\begin{subequations}
    \begin{align}
        &\kappa(H) \gtrsim \frac{N+\lambda}{\frac{3N}{4R} + \lambda}, 
        \qquad \forall R \lesssim \frac{3N}{4},
        \label{eq:cond number radius lessthan}
        \\
        &\kappa(H) = \frac{N+\lambda}{\lambda}, 
        \qquad \forall R \gtrsim \frac{3N}{4}.
    \end{align}
    \label{eq:cond number radius}
\end{subequations}
More generally, if the ratio of the number of pixels in a projected image and the 
number of voxels in a volume at a given resolution $R$ is $p(R)$, then the bounds 
in~\eqref{eq:cond number square lessthan} and~\eqref{eq:cond number radius lessthan} 
can be written as:
\begin{equation}
    \kappa(H) \geq \frac{N + \lambda}{p(R) N + \lambda}, 
    \quad \forall R \quad \text{such that} \quad \frac{1}{p(R)} \leq N.
    \label{eq:cond number p}
\end{equation}

\begin{remark}
    While the counting argument above shows that the condition number is large when 
    there are more voxels to reconstruct than pixels in all the particle images, in 
    practice, the condition number grows fast with the resolution due to an additional 
    factor. Specifically, in light of the Fourier slice theorem, each image is used 
    to reconstruct the voxels corresponding to a slice through the volume passing 
    through the center of the coordinate axes. Therefore, the large condition number 
    of the matrix H is also a consequence of the fact that the elements of $H$ 
    corresponding to low-frequency voxels are reconstructed using pixel values in 
    most images, while the elements of H corresponding to high-frequency voxels are 
    ``seen'' by fewer pixels in the particle images. Each new iteration will provide 
    more information to the low-frequency voxels than to the high-frequency ones 
    (relative to the total number of low-frequency voxels and high-frequency voxels,
    respectively), which leads to errors being amplified (or corrected) at different 
    rates when solving the inverse problem. Lastly, this problem is exacerbated by 
    preferred orientations of the particles: the orientation angles often do not cover 
    SO(3) evenly in real datasets, causing the Fourier transform of the volume to miss 
    entire slices.
\end{remark}

In Figure~\ref{fig:cond number plots}, we illustrate the statement above for the
setting of this section, specifically with nearest-neighbor interpolation in the
projection operators and no CTF.
In panel (a), we show the lower bound on $\kappa(H)$ given in~\eqref{eq:cond number p} 
as a function of the radius in the Fourier space, as well as the condition number 
for a dataset of $N=10000$ particle images with uniformly sampled orientations 
with $R$ ranging from $1$ to $304$ voxels at intervals of $16$ and $\lambda = 10^{-8}$.
The condition number grows faster than the lower bound due to the effects 
described in the previous paragraph.
To further illustrate the relationship between the number of images and the condition 
number, we show in panels (b,c) of Figure~\ref{fig:cond number plots} the Hessian $H$ 
when using nearest-neighbor interpolation and no CTF when the dataset contains 
$N=5$ images (panel (b)) and when the dataset contains $N=100$ images (panel (c)). 
This shows how the particle images contribute to a larger fraction of the voxels 
close to zero than those at a large Fourier radius.

In light of the dependence of the rate of convergence of SGD on the condition 
number of the Hessian $H$ given in Theorem~\ref{thm:sgd convergence} and 
equation~\eqref{eq:min iters}, Figure~\ref{fig:cond number plots}(a) suggests 
that the number of iterations required to reach a certain error grows exponentially 
with the resolution. 
Since the root cause is the large condition number at high resolution, we will
address this issue by preconditioning the SGD iterations, specifically by using
an approximation of the diagonal $D \in \mathbb{R}^{M_v \times M_v}$ of $H$:
\begin{equation}
    v^{(k+1)} = v^{(k)} - \eta_k D^{-1} \nabla f(v^{(k)}).
    \label{eq:sgd step}
\end{equation}
\begin{remark}
    For fixed $\sigma, \tau$, taking $\lambda=\sigma^2/\tau^2$ and known poses, 
    the matrix whose inverse appears in the M step of the EM algorithm in 
    equation~\eqref{eq:em v} is the full Hessian matrix $H$ in~\eqref{eq:hessian} 
    of the loss function, and so the M step becomes: 
    \begin{equation*}
        v^{(k+1)} = \left(
            \sum_{i=1}^N P_i^* P_i + \lambda I_{M_v}
        \right)^{-1} 
        \left(
            \sum_{i=1}^N P_i^* x_i
        \right)
        = H^{-1} 
        \left(
            \sum_{i=1}^N P_i^* x_i,
        \right)
    \end{equation*}
    Therefore, EM implicitly solves the conditioning issue 
    that is problematic for SGD, and in our preconditioning approach, we aim to 
    approximate, using mini-batches, the diagonal part of
    this gradient scaling that EM applies at every iteration.
\end{remark}
With the facts above regarding the condition number of the Hessian of the loss function,
we now proceed to estimate the diagonal preconditioner for this matrix.

\begin{figure}[ht!]
    \centering
    \begin{subfigure}[b]{0.4\textwidth}
        \centering
        \includegraphics[width=\textwidth]{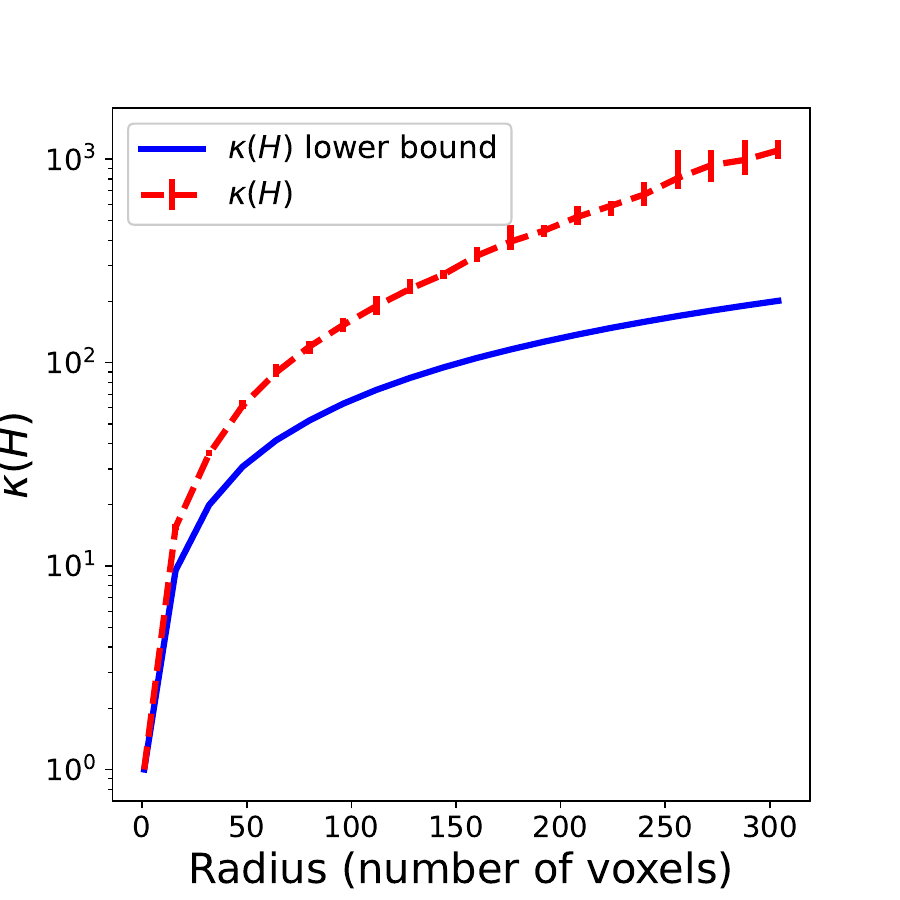}
        \caption{$\kappa(H)$ and lower bound}
    \end{subfigure}
   
    \begin{subfigure}[b]{0.4\textwidth}
        \centering
        \includegraphics[width=\textwidth]{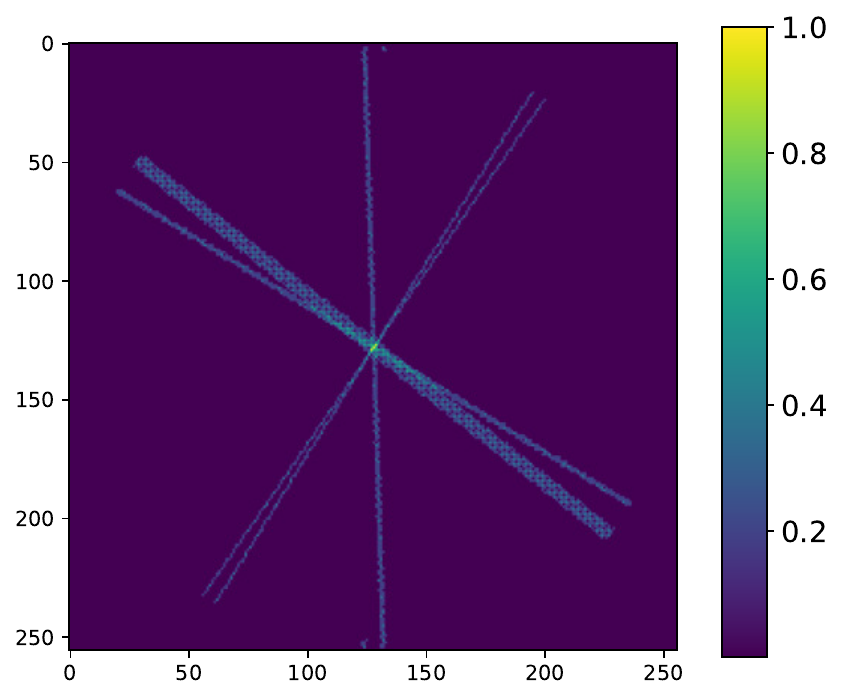}
        \caption{$z=0$ slice of $H$ for $N=5$}
    \end{subfigure}
    %\hfill
    %
    \begin{subfigure}[b]{0.4\textwidth}
        \centering
        \includegraphics[width=\textwidth]{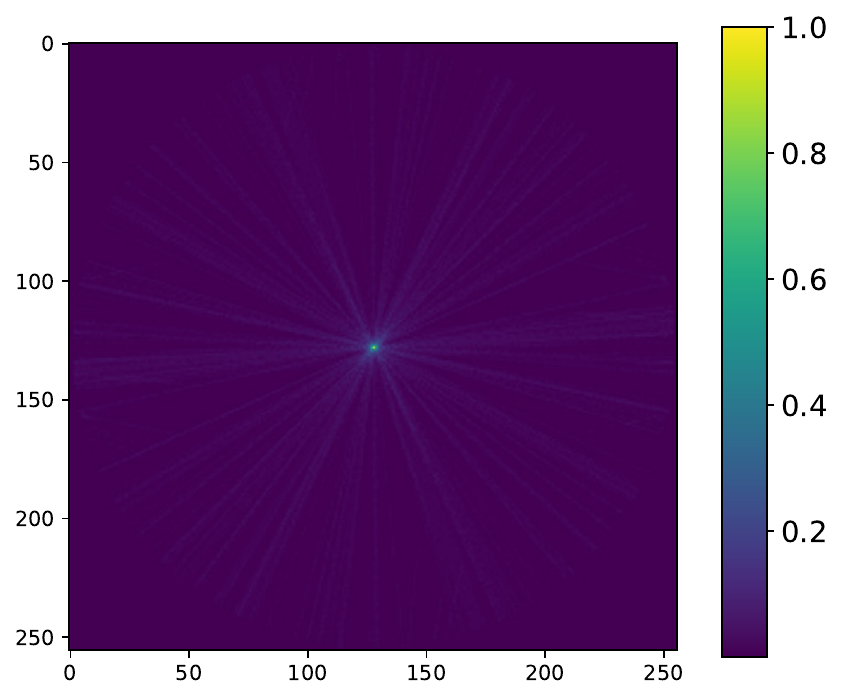}
        \caption{$z=0$ slice of $H$ for $N=100$}
    \end{subfigure}
    \caption{
        The condition number $\kappa(H)$ for nearest-neighbor interpolation and
        no CTF.
        (a) Lower bound (blue, solid) and average value (red, dashed) of $\kappa(H)$ 
        for $N = 10000$ images, Fourier radius of $R$ voxels in the range 
        $\{1, 16, 32, 48, 64, \ldots, 304\}$, and $\lambda=10^{-8}$.
        The lower bound is computed using equation~\eqref{eq:cond number p}, where
        $p(R)$ is computed empirically using spherical and circular shells of radius $R$. 
        To compute the average value of $\kappa(H)$ for each $R$, 10 sets of $N$ 
        uniformly distributed orientations are generated and the average condition 
        number is plotted, as well as error bars showing the minimum and maximum 
        values of the condition number for these sets of orientations.
        (b,c) The $z=0$ plane of the diagonal of the nearest-neighbor Hessian matrix 
        $H$ reshaped as a 3D volume, for $N=5$ and $N=100$ images, respectively, 
        and grid size $256$.}
    \label{fig:cond number plots}
\end{figure}

\subsection{Computing the preconditioner}
\label{sec:hutch}

We aim to obtain an approximation of the diagonal of the Hessian matrix $H$ and 
use it to precondition SGD, which is equivalent to preconditioning the linear 
system~\eqref{eq:lin sys} using the Jacobi preconditioner~\cite{trefethen_numerical_1997}.
Motivated by algorithms such as AdaHessian and OASIS in the machine learning 
literature~\cite{yao_adahessian_2021,jahani_doubly_2022}, we estimate the diagonal 
of the Hessian using Hutchinson's diagonal estimator~\cite{bekas_estimator_2007}
stated in~\eqref{eq:hutch}.

Estimating the diagonal of $H$ using~\eqref{eq:hutch} has two practical advantages. 
First, for any function $f$, applying~\eqref{eq:hutch} only requires computing 
Hessian-vector products, rather than forming the full Hessian matrix. 
Indeed, for a function $f$, a Hessian-vector product is computed efficiently 
using Jacobian-vector products and automatic differentiation as follows:
\begin{equation*}
    %\nabla^2 f(v)z = \nabla[v \to \nabla f(v)^T z].
    (v, z) \mapsto \nabla\left(\nabla f(v)\right) z
\end{equation*}
Second, the computation can be split into mini-batches so that, at each iteration,
only the current mini-batch of images is used for the Hessian-vector product computation.

The update of the preconditioner at the current iteration obtained using the current 
mini-batch is combined with the estimated preconditioner from the previous iteration 
using an exponential average as done, for example, in Adam~\cite{kingma_adam_2015}, 
AdaHessian~\cite{yao_adahessian_2021} and OASIS~\cite{jahani_doubly_2022}.
In addition, to take advantage of the fact that the Hessian $H$ in~\eqref{eq:hessian}
of the objective function~\eqref{eq:optim cp} is independent of the current 
iterate when the orientations are known, the exponential average is taken between 
the estimated preconditioner at the previous iteration and the estimated diagonal 
using all the samples of Rademacher vectors $z$ drawn up to the current iteration. 
However, in the more general reconstruction problem with unknown pose 
variables~\eqref{eq:posterior}, only the current update would be used.
Starting with the identity matrix as the initial estimate, $D^{(0)} = I_{M_v \times M_v}$, 
the update rule for the diagonal estimator $D^{(k)}$ at iteration $k$ is:
\begin{subequations}
    \begin{align}
        D_{avg}^{(k)} &= \frac{1}{k}\left( z^{(k)} \odot Hz^{(k)}\right),
        \label{eq:hutch update}
        \\
        D^{(k)} &= \beta D^{(k-1)} + (1-\beta) D_{avg}^{(k)},
        \label{eq:exp avg update}
    \end{align}
    \label{eq:diag est rule}
\end{subequations}
where $\beta \in (0,1)$ and the Hessian-vector product is computed using the 
current mini-batch. 
An example of the convergence of this estimate over $100$ batches of $1000$ images 
each and a total number of $N=30000$ images is shown in Figure~\ref{fig:diag est}.
\begin{figure}[ht]
    \centering
    \begin{subfigure}[b]{0.4\textwidth}
        \centering
        \includegraphics[width=\textwidth]{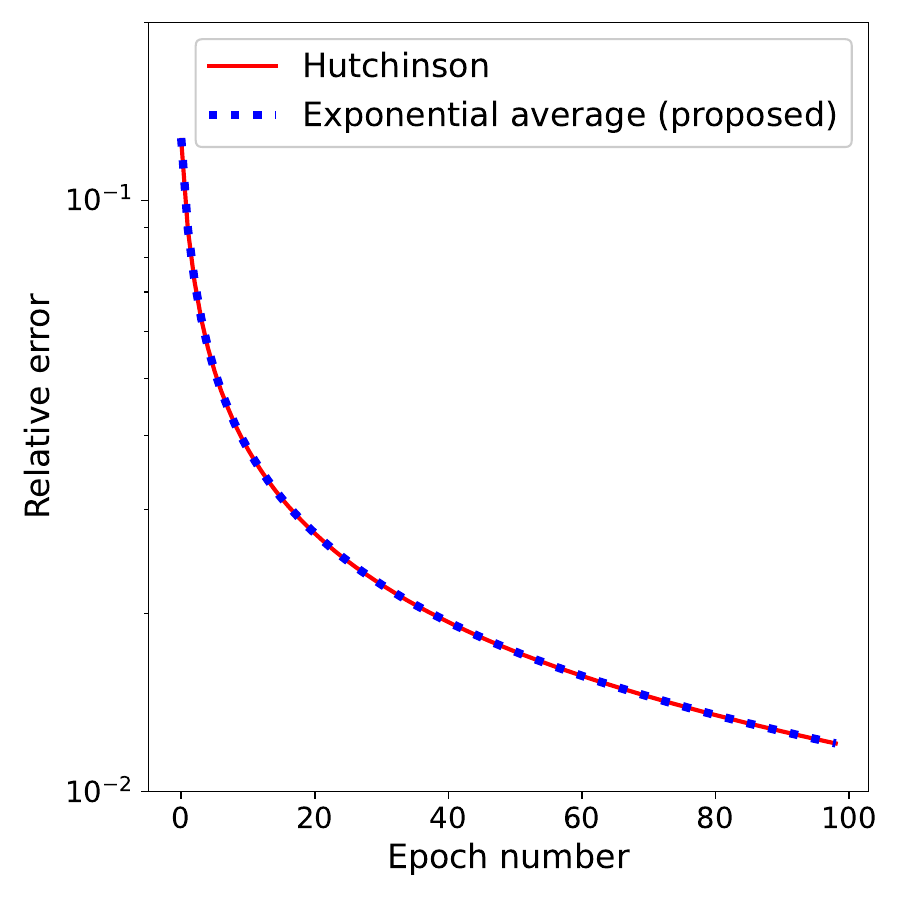}
        \caption{Relative $\ell_2$ error}
    \end{subfigure}
    %\hfill
    %
    \begin{subfigure}[b]{0.4\textwidth}
        \centering
        \includegraphics[width=\textwidth]{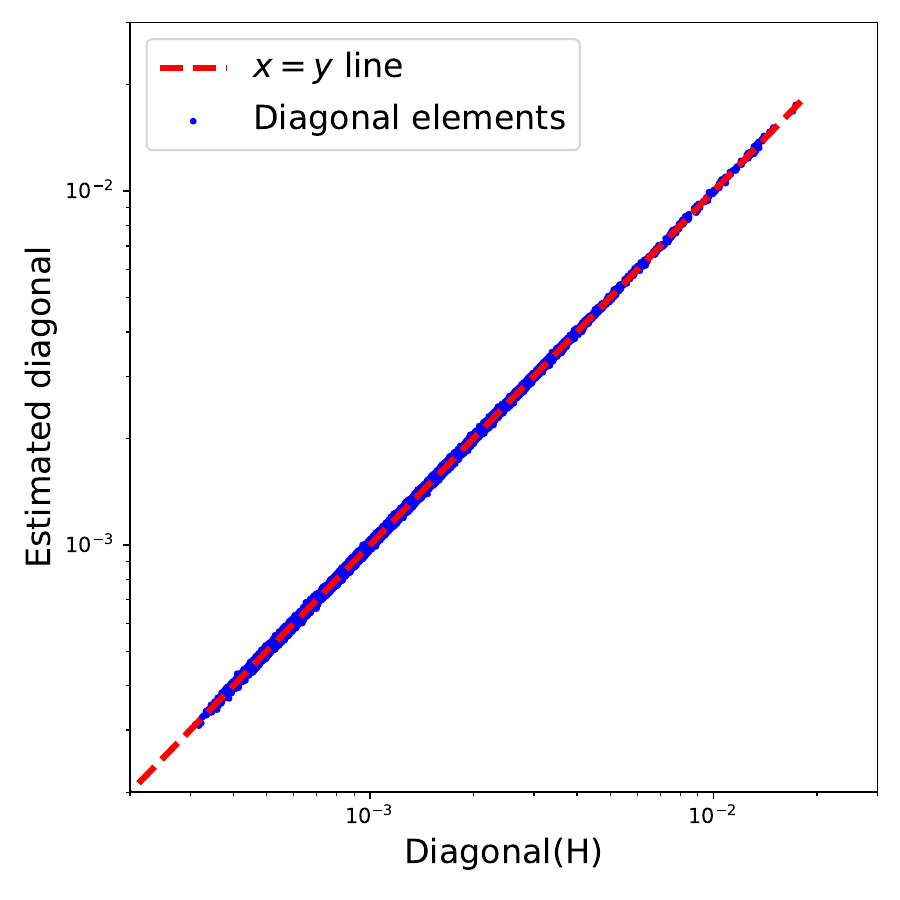}
        \caption{Individual elements error}
    \end{subfigure}
    \caption{Estimating the diagonal of $H$ using~\eqref{eq:diag est rule}, where
        $H$ is the Hessian matrix of the loss function $f$ in~\eqref{eq:optim cp}
        with trilinear interpolation (and therefore, $H$ is not diagonal).
        (a) Relative $\ell_2$ error of the estimated diagonal of $H$ using 
        Hutchinson's diagonal estimator~\eqref{eq:hutch} and the proposed exponential 
        average~\eqref{eq:diag est rule} between an initial estimate $D^{(0)}$ 
        (here the identity matrix) and the average given by Hutchinson's estimator. 
        (b) The individual elements of the estimated diagonal plotted against their 
        true values. The accuracy in the smaller elements is lower, as expected 
        for high-frequency elements not mapped by the projection operators of 
        many particle images, i.e. elements at high-frequency indices $j$ with 
        small $|\Omega_j|$.}
    \label{fig:diag est}
\end{figure}

\begin{remark}
    When the projection operator uses nearest-neighbor interpolation, i.e. 
    $P_i := P_{\phi_i}^{nn}$, and therefore the Hessian matrix $H$ is diagonal, 
    Hutchinson's estimator with Rademacher samples~\eqref{eq:hutch} computes the 
    exact diagonal of $H$ using a single sample vector $z$ or, if computed using 
    mini-batches, after a single epoch.
\end{remark}

\subsection{Thresholding of the estimated preconditioner}
\label{sec:thresholding}

The combination of variance
of the sampled gradient using mini-batches and error in the diagonal approximation
of the Hessian, especially in the early iterations, can lead to highly amplified 
errors in the current iterate $v^{(k)}$. This is particularly problematic in the 
case of particles with preferred orientations.

In general, this problem can be avoided by using variance-reduced stochastic gradient 
methods~\cite{johnson_accelerating_2013,defazio_saga_2014}, which require computing 
full gradients at a subset of the iterations or storing previously computed gradients. 
However, in a typical cryo-EM setting with large datasets, this has a prohibitive 
computational cost.

Instead, we propose a simple solution that leverages the particular structure of 
the cryo-EM reconstruction problem. Given the specific structure of the matrices
$P_i$ and $P_i^*P_i$ (see Definition~\ref{def:proj} and equation~\eqref{eq:pp})
and the fact that the projection operators perform slicing in the Fourier domain
(according to the Fourier slice theorem), the preconditioned SGD iteration~\eqref{eq:sgd step} 
updates the elements of $v^{(k)}$ corresponding to high-frequency voxels at a
lower rate using information in the particle images compared to the low-frequency
voxels. Similarly, the elements of the estimated preconditioner $D^{(k)}$ corresponding 
to high-frequency voxels in the volume have small values, close to the regularization 
parameter $\lambda$, while the elements corresponding to low-frequency voxels have 
magnitudes that reflect the large number of images that contribute to the reconstruction 
of the voxels (see the discussion in the paragraph below equation~\eqref{eq:cond number p}), 
and are scaled by the CTF at that particular resolution. 
In light of~\eqref{eq:H diag}, this is due to $|\Omega_j|$ being large for low-frequency
elements and small for high-frequency elements of the diagonal of $H$.

This knowledge can be incorporated into the reconstruction algorithm in two ways.
One approach is to tune the regularization parameter $\lambda$ so that it balances 
the small entries in the preconditioner with the ability to obtain good convergence 
at high resolution. Alternatively, the same effect can be achieved without interfering 
with the regularization term by thresholding the smallest elements of the preconditioner, 
with the benefit of using the structure of the measurement operator in the preconditioner 
while allowing the freedom to chose the regularization term and parameter by other 
means, for example as done in RELION~\cite{scheres_relion_2012}.
Here, we opt for the latter approach, namely thresholding the elements of 
the preconditioning matrix at the current iteration $D^{(k)}$ below a certain 
value $\alpha$, chosen as follows.

Let $C_i(r)$ be the radially symmetric CTF corresponding to the $i$-th image, as 
a function of the Fourier radius $r \in [0, R]$, and let $P_x(r)$ and $P_v(r)$ 
be the number of pixels in a 2D Fourier shell at radius $r$ and the number of voxels 
in a 3D Fourier shell at radius $r$, respectively.
Following the steps in Section~\ref{sec:cond} with non-trivial CTFs and assuming 
that the particle orientations are uniformly distributed (the details are omitted 
for brevity), the expected value of an element of the Hessian matrix $H$ at the 
maximum Fourier radius $R$ is:
\begin{equation}
    \bar{H} = \frac{P_x(R)}{P_v(R)} \cdot \sum_{i=1}^N |C_i(R)|^2 + \lambda,
    \label{eq:thresh value}
\end{equation}
Setting the threshold of the smallest elements of $D^{(k)}$ to $\alpha = \bar{H}$,
the final preconditioner at iteration $k$ is defined as: 
\begin{equation}
    \hat{D}^{(k)}_{jj} = \max\{|D^{(k)}_{jj}|, \alpha \}, 
    \quad \text{for all} \quad j = 1, \ldots, M_v.
    \label{eq:thresh}
\end{equation}

\subsection{The algorithm}
\label{sec:algorithm}

With the preconditioner estimation approach proposed in Section~\ref{sec:hutch}
and the thresholding described in Section~\ref{sec:thresholding}, we present the 
full algorithm in Algorithm~\ref{alg:precon sgd}. 
No warm start is required for the approximation of the diagonal preconditioner, 
which is initialized with the identity matrix $I_n$ and estimated iteratively;
the running average $D_{avg}$ computed as part of the diagonal 
approximation~\eqref{eq:diag est rule} is initialized as the 
zero $n \times n$ matrix $O_n$.
The step size $\eta^{(k)}$ at each iteration is set using the stochastic Armijo 
line-search~\cite{vaswani_adaptive_2021}, namely the largest step size $\eta$
is sought that satisfies the condition:
\begin{equation}
    f_{\Ic_k}\left(v^{(k)} - \eta D^{-1} \nabla f_{\Ic_k}(v^{(k)}) \right) 
    \leq f_{\Ic_k}(v^{(k)}) - c \cdot \eta \|\nabla f_{\Ic_k}(v^{(k)}) \|^2_{D^{-1}},
    \label{eq:stochastic armijo}
\end{equation}
where $\mathcal{I}_k$ is the index set of the current mini-batch, $D$ is
the preconditioner, and $c \in (0,1)$ is a hyper-parameter.
Condition~\eqref{eq:stochastic armijo} ensures that a sufficient decrease in the 
objective function is attained over the current mini-batch at every iteration

\begin{algorithm}[h]
    %\caption{Preconditioned SGD}
    \caption{}
    \label{alg:precon sgd}
    \begin{algorithmic}[1]
        \Require $v^{(0)}$, $\eta^{(0)}$, $c$,
        \State Let $\alpha = \bar{H}_R$ as in~\eqref{eq:thresh value}
        \State $D_0 = I_n $
        \State $D_{avg} = O_n$ 
        \For{$k = 1, 2, \ldots$} 
            \State Draw index set $\Ic_k$ and Rademacher(0.5) vector $z^{(k)}$
            \State $D_{avg} = \frac{k-1}{k} D_{avg} + \frac{1}{k} z^{(k)} \odot \nabla^2 f_{\Ic_k}(v^{(k)}) z^{(k)}$ 
            \Comment{Hutchinson's estimator~\eqref{eq:hutch update}}
            \State $D^{(k)} = \beta D^{(k-1)} + (1-\beta) D_{avg}$
            \Comment{Estimate update~\eqref{eq:exp avg update}}
            \State $\hat{D}^{(k)}_{jj} = \max\{|D^{(k)}_{jj}|, \alpha\}, 
                \quad \forall j=1,\ldots,M_v$
            \Comment{Preconditioner thresholding~\eqref{eq:thresh}}
            \State $\eta^{(k)} = \eta^{(k-1)}$
            \State $v^{(k+1)} = v^{(k)} - \eta^{(k)} (\hat{D}^{(k)})^{-1} \nabla f_{\Ic_k}(v^{(k)})$
            \Comment{Stochastic gradient step~\eqref{eq:sgd step}}
            \While{$f_{\Ic_k}(v^{(k+1)}) > f_{\Ic_k}(v^{(k)}) - c \cdot \eta^{(k)} \|\nabla f_{\Ic_k}(v^{(k)})  \|^2_{{(\hat{D}^{(k)})}^{-1}} $}
            \Comment{Line-search~\eqref{eq:stochastic armijo}}
                \State $\eta^{(k)} = \eta^{(k)} / 2$
                \State $v^{(k+1)} = v^{(k)} - \eta^{(k)} \hat{D}^{-1} \nabla f_{\Ic_k}(v^{(k)})$
            \EndWhile
        \EndFor 
    \end{algorithmic}
\end{algorithm}

%\newpage

\section{Numerical experiments}
\label{sec:experiments}

In this section, we present numerical experiments\footnote{
    The code for reproducing the numerical experiments is available at \\
    \url{https://github.com/bogdantoader/simplecryoem}.
} that demonstrate the arguments
given in this article regarding the convergence of the SGD algorithm for 
high resolution cryo-EM reconstruction. In particular, the numerical experiments
in this section further verify two claims:

\begin{itemize}
    \item\textbf{Claim 1:} The condition number $\kappa(H)$ of the Hessian of the
        loss function is large at high resolution, which leads to slow convergence
        of the SGD algorithm.

    \item\textbf{Claim 2:} Preconditioning SGD using the diagonal preconditioner as 
        estimated using the tools described in Sections~\ref{sec:hutch}-\ref{sec:algorithm}
        leads to improved convergence at high resolution.
\end{itemize}

\subsection{Setup}

The dataset used in this section is derived from the Electron Microscopy Public 
Image Archive database entry EMPIAR-10076 of the bacterial ribosome, where 
$N=30000$ images of $192 \times 192$ pixels are selected and whose pose variables 
have been computed using RELION\footnote{
    The parameters were first obtained from the CryoDRGN tutorial~\cite{zhong_cryodrgn_2021},
    converted to RELION format using Pyem~\cite{asarnow_pyem_2019} and then further processed
    in RELION~\cite{scheres_relion_2012}. 
    Specifically, one round of 3D refinement and one round of 3D classification 
    with 8 classes were performed, followed by one final 3D refinement of one
    of the resulting homogeneous classes, consisting of $N=33026$ particles, 
    to a resolution of $4.2$ \AA. $N=30000$ particles were selected at random
    from the resulting homogeneous class.
    While the complete removal of all continuous heterogeneity is not relevant 
    for this paper and therefore has not been attempted, the resulting dataset 
    is still a good approximation of the EMD-8455 and EMD-8456 discrete classes 
    associated with the EMPIAR-10076 entry.
}. The condition number of the reconstruction 
problem for this dataset with the computed pose variables is shown in 
Figure~\ref{fig:cond number plots real data}(a) for increasing values of the 
maximum Fourier radius $R$. Nearest-neighbor interpolation has been used for this
figure, and therefore the condition number has been computed using~\eqref{eq:cond H}.
The condition number grows quickly with the resolution, and at the maximum radius 
$R$ corresponding to the grid side length $M=192$, it is of the order of $10^2$.
The bound in Proposition~\eqref{prop:cond} and 
equations~\eqref{eq:cond number radius}-\eqref{eq:cond number p}, shown as the solid 
blue line in Figure~\ref{fig:cond number plots real data}(a), holds for the condition
number when no CTF is included, shown as the red, dashed line. 
While the theory in Section~\ref{sec:main} does not handle the case when CTF is used, we will show empirically in Section~\ref{sec:experiments results} that the same trend of the ill-conditionning increasing with the resolution holds when the particle images, extracted from true micrographs, are corrupted by the CTF.
To see why the theory does not apply, note that because the CTF has magnitude less than one at the origin and it oscillates around zero, 
the numerator in~\eqref{eq:cond H} is smaller than in the 
case when no CTF is used, and therefore the condition number computed in this setting, 
shown in the green dashdot line, is not always larger than the theoretical lower bound.
Figure~\ref{fig:cond number plots real data}(b), showing a one-dimensional cross-section 
through the diagonal of the Hessian both with and without CTF in the same setting as 
in Figure~\ref{fig:cond number plots real data}(a), explains the shape of the
green curve in Figure~\ref{fig:cond number plots real data}(a): both the minimum
and the maximum values of $\text{diag}(H)$ occur in the interval $[0,20]$ (in voxels).
The effect of the CTF is further illustrated in Figure~\ref{fig:cond number plots real data}(c), 
in the two-dimensional cross-section of the diagonal of the Hessian matrix $H$ 
in the case when trilinear interpolation is used. 
While the CTF leads to a Hessian condition number that is lower than the one without CTF, the difficulty of the underlying inverse problem is not decreased: as the CTF removes information from the images, the difficulty of the reconstruction problem is actually increased.

To verify the two claims on the dataset described above, we run three algorithms:
SGD with no preconditioner; SGD with a preconditioner that has been precomputed 
in advance using~\eqref{eq:diag est rule} over 1000 epochs with batch size 1000; 
and Algorithm~\ref{alg:precon sgd}, namely SGD with a preconditioner estimated during the refinement process over 10 epochs.
All algorithms are initialized with the same random volume obtained by sampling
each voxel value from a complex normal distribution,
and use the stochastic Armijo line-search~\eqref{eq:stochastic armijo}
for the step size adaptation. While the theory in Section~\ref{sec:cond} applies 
to nearest-neighbor interpolation, where the matrices $P_i^* P_i$ are diagonal
(and therefore the optimization problem can be solved easily by other methods),
in the numerical experiments presented in this section, the projection operators
are implemented using trilinear interpolation, which is often used in practice. 
In this case, the Hessian matrix $H$, whose diagonal we estimate and use as a 
preconditioner, is no longer diagonal.

%\newpage 

\begin{remark}
    The convergence of Algorithm~\ref{alg:precon sgd} is expected to be at most as 
    good as the convergence of SGD with the precomputed preconditioner. 
    This is due to the fact that the preconditioner has been computed using 1000
    epochs of~\eqref{eq:diag est rule} and this preconditioner is then used
    from the start of the SGD algorithm, while Algorithm~\ref{alg:precon sgd}
    starts with the identity matrix as the preconditioner and estimates it during
    the 10 epochs (using~\eqref{eq:diag est rule} as well). 
    If the particle poses are known and fixed, one can simply precompute 
    the preconditioner corresponding to the forward operator with the given poses in advance.
    However, the aim of the proposed method for estimating the
    preconditioner is to be incorporated into a reconstruction algorithm that
    estimates the poses in addition to the volume. By doing so, the optimal 
    preconditioner changes when pose variables are updated, making it impossible
    to precompute a good preconditioner in advance. Our aim is to show that over
    only a small number of epochs (10 in our numerical experiments), 
    the proposed method computes a good enough preconditioner that allows obtaining 
    a solution to the inverse problem at high-frequencies.
\end{remark}

To evaluate and compare the performance of the three algorithms, we first compute a numerical ground truth solution to the optimization problem~\eqref{eq:optim cp} using the L-BFGS algorithm~\cite{liu_limited_1989}, which we run for 1000 iterations. 
This approach to obtaining a ground truth benefits from the convexity of problem~\eqref{eq:optim cp}, as well as the good convergence properties of the L-BFGS algorithm due to its estimation of the Hessian of the loss function. 
By running it for a number of iterations much larger than the number of epochs for which we would run SGD in a practical setting in cryo-EM reconstruction, we ensure that the L-BFGS solution is a good baseline against which to evaluate the SGD solutions.
We show the convergence of L-BFGS on the given dataset and problem~\eqref{eq:optim cp} in Figure~\ref{fig:cond number plots real data}(d).

In the runs of the three SGD algorithms, we compute the value of the loss function in~\eqref{eq:optim cp} after each epoch (a full pass through the dataset), as well as the Fourier Shell Correlation (FSC) of each reconstruction with the L-BFGS ground truth solution.

The FSC, a standard error measure in the cryo-EM literature, is the cross-correlation 
coefficient between two volumes across three-dimensional shells in the Fourier 
domain~\cite{harauz_exact_1986}. Specifically, given two volumes $u$ and $v$ in 
the Fourier space, the FSC at radius $r$ from the origin is defined as
\begin{equation}
    \text{FSC}(r) = \frac{
        \sum_{\ell \in S_r} u_{\ell} v_{\ell}^*
    }{
        \sqrt{(\sum_{\ell \in S_r} |u_{\ell}|^2)(\sum_{\ell \in S_r} |v_{\ell}|^2)}
    },
    \label{eq:fsc}
\end{equation}
where $S_r$ is the set of Fourier voxels in the spherical shell at radius $r$.
While the FSC is used in the cryo-EM literature as a measure of the resolution or consistency of a reconstruction, here it will be used differently. Instead, we will compute the FSC between the iterates of each SGD algorithm and the L-BFGS ground truth volume in order to understand how different frequency bands of the structures are converging at different rates.

\subsection{Results}
\label{sec:experiments results}

Figure~\ref{fig:results} shows the results obtained using the three algorithms.
In panel (a), we see the convergence of the loss function significantly improved
when using the precomputed preconditioner (red) over not using a preconditioner (blue),
verifying Claim 1 above. 
When estimating the preconditioner during refinement
using Algorithm~\ref{alg:precon sgd}, the value of the loss (green, solid) is between 
those of the previous two algorithms and approaches the fully preconditioned SGD
as the preconditioner progressively becomes more accurate, verifying Claim 2.
For reference, the black dashed line shows the loss function value at the final volume obtained using L-BFGS and used as a ground truth reconstruction. 
As expected, this is lower than all of the three SGD algorithms that we evaluate in this section.
In panel (b), we show the FSC between the final reconstructions of each of the
three algorithms and the ground truth reconstruction. While at low resolution,
the FSC value is high for all algorithms, showing good convergence, at high resolution
the FSC degrades for SGD with no preconditioning (blue), while the two preconditioned
algorithms have a higher value, closer to one. 

For further insight into the difference between SGD without a preconditioner and 
SGD with a precomputed preconditioner at low and high resolution, as well as 
how the estimated preconditioner behaves in relation to them, we show
in Figure~\ref{fig:fsc epoch} the FSC with the ground truth for the three algorithms
at specific Fourier shells and across epochs.
Panel (a) shows the FSC for a low-frequency Fourier shell, panel (b) shows a 
medium-frequency shell, and panel (c) shows a high-frequency shell. 
At low resolution, all three algorithms are almost indistinguishable, while at medium 
resolution, the non-preconditioned SGD shows slower convergence than the preconditioned 
SGD, with SGD with the the estimated preconditioner quickly approaching the accuracy of 
the fully preconditioned SGD. 
The advantage of using a preconditioner is seen most clearly at high resolution,
where the FSC of non-preconditioned SGD is much lower than the other two algorithms. 
SGD with an estimated preconditioner converges more slowly than at medium resolution, 
but eventually approaches the FSC of SGD with a precomputed preconditioner. 
The small difference in the FSC at the last epoch
between the two preconditioned algorithms is likely due to a combination of factors,
for example the limited accuracy of the preconditioner estimate after only 10 epochs.
Lastly, we show in Figure~\ref{fig:fsc shell epoch} a complete overview of
the shell-wise FSC for each algorithm using a heat map of the correlation as a 
function of the Fourier shell number and the epoch number.

In Figure~\ref{fig:maps}, we show a qualitative comparison between the outputs of the three SGD algorithms and the ground truth volume to illustrate the effect of the preconditioner on the high-resolution details in the reconstructions.
The figure shows four different views of the volumes (one view on each row), where the volumes are displayed as the isosurface at six standard deviations above the mean voxel value of each volume. 
Firstly, we see that the main differences between the ground truth L-BFGS volume (first column, grey) and the volume obtained using SGD with no preconditioner (second column, blue) are in the high-resolution details, where the former shows more defined fine-scale features than the latter.
Moreover, both volumes obtained using the preconditioned SGD (third column, red for the precomputed preconditioner and fourth column, green for the estimated preconditioner using Algorithm~\ref{alg:precon sgd}) show a similar level of high-resolution detail to the ground truth volume.
Secondly, there is no significant visual difference between the volumes obtained using the precomputed and estimated preconditioners.
These conclusions drawn from Figure~\ref{fig:maps} are consistent with the quantitative results shown in Figures~\ref{fig:results}-\ref{fig:fsc shell epoch}, especially regarding the convergence of the SGD algorithms at low resolution (where the preconditioner does not have a significant effect) and at high resolution (where the preconditioner leads to the differences between the reconstructions), as shown in Figure~\ref{fig:fsc epoch} and Figure~\ref{fig:fsc shell epoch}.

The aim of the figures in this section is to show, either quantitatively (Figures~\ref{fig:results}-\ref{fig:fsc shell epoch}) or qualitatively (Figure~\ref{fig:maps}), the differences between the resulting volumes and how the preconditioner leads to a different reconstruction that is closer to the numerical ground truth obtained by solving the optimization problem~\eqref{eq:optim cp} to higher accuracy.
The results in this section (and more generally in this paper) do not refer to a measure of the overall quality of a cryo-EM map, as this is dependent on multiple factors, both in terms of how the map is obtained and in terms of how such a measure of the map quality is computed. In particular, obtaining a ``good'' map involves multiple steps, including estimating the noise in the images, accurately estimating the pose variables, appropriately regularizing and filtering the maps (i.e. by splitting the data into half sets, computing independent reconstructions, and obtaining a Wiener filter from the FSC between the two reconstructions, that is then applied to the final map) and solving the volume reconstruction problem to high accuracy. Measuring the quality of a cryo-EM map, likewise, involves computing the FSC between independent half set reconstructions and using the appropriate FSC threshold to determine the final resolution (and hence the quality) of the map.
As this paper only addresses the volume reconstruction problem (given fixed values of the other parameters such as pose variables, noise level and regularization parameter), the results in the current section only show the contribution of the preconditioner to the accuracy of the solution to the volume reconstruction problem~\eqref{eq:optim cp} and not how the map quality would improve if such a preconditioner was used in the full cryo-EM image processing pipeline that also estimates the other parameters together with accurately measuring the final resolution based on half-set FSC curves. 
Thus, the usual FSC threshold values of $0.143$ or $0.5$ used in the cryo-EM literature~\cite{rosenthal_optimal_2003} are not meaningful in our plots showing FSC curves.

Finally, in light of equation~\eqref{eq:min iters}, the number of iterations
or epochs required for a gradient-based algorithm to achieve an accuracy $\epsilon$
scales linearly with the condition number of the Hessian matrix $H$. 
We show in Figure~\ref{fig:num iters real data} the impact of the condition number
on the number of epochs required by each of the three algorithms to reach 
a certain FSC value in our experiments, for each Fourier shell. As the condition
number of the preconditioned algorithms scales better with the resolution, the 
plots show the number of iterations growing more slowly with the resolution 
for the preconditioned algorithms than in the non-preconditioned case.

\begin{figure}%[H]
    \centering
    \begin{subfigure}[b]{0.4\textwidth}
        \centering
        \includegraphics[width=\textwidth]{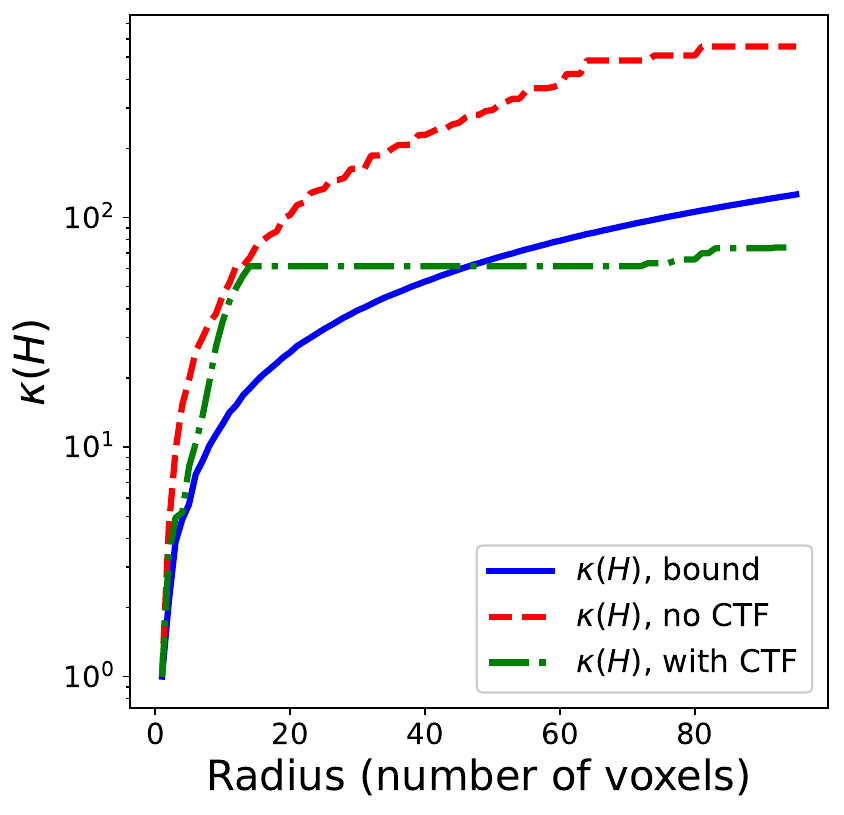}
        \caption{$\kappa(H)$ and the lower bound}
    \end{subfigure}
    \begin{subfigure}[b]{0.4\textwidth}
        \centering
        \includegraphics[width=\textwidth]{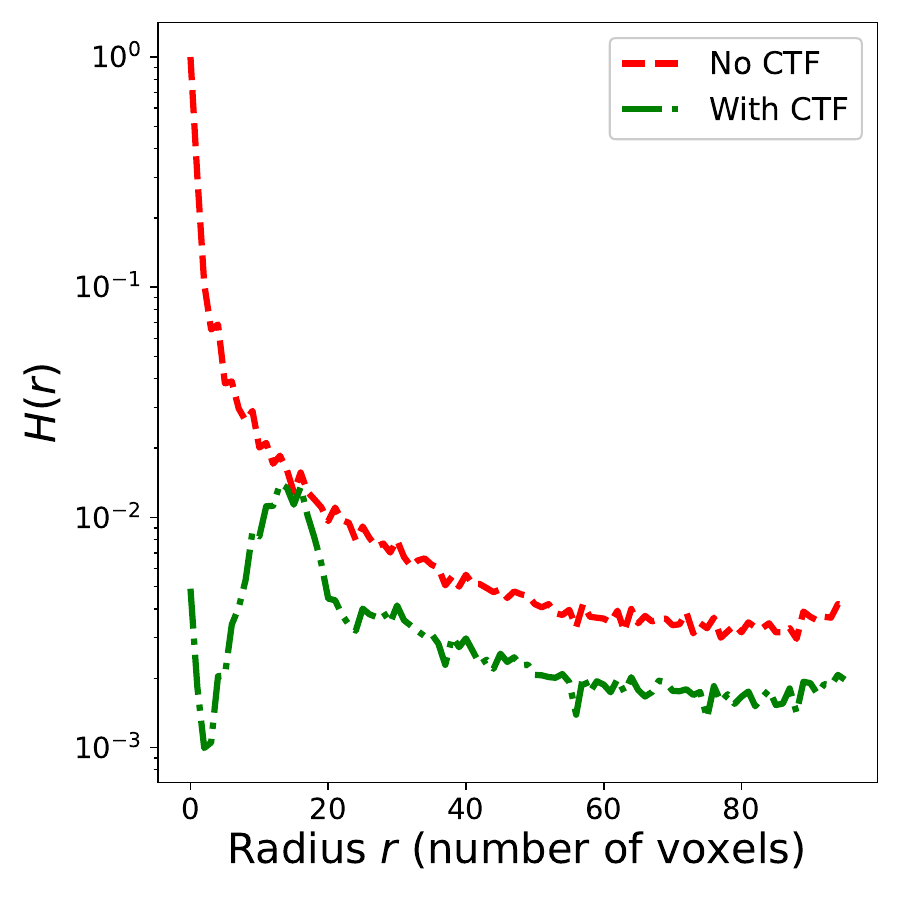}
        \caption{1D cross-section of $\text{diag}(H)$}
    \end{subfigure}
    \begin{subfigure}[b]{0.4\textwidth}
        \centering
        \includegraphics[width=\textwidth]{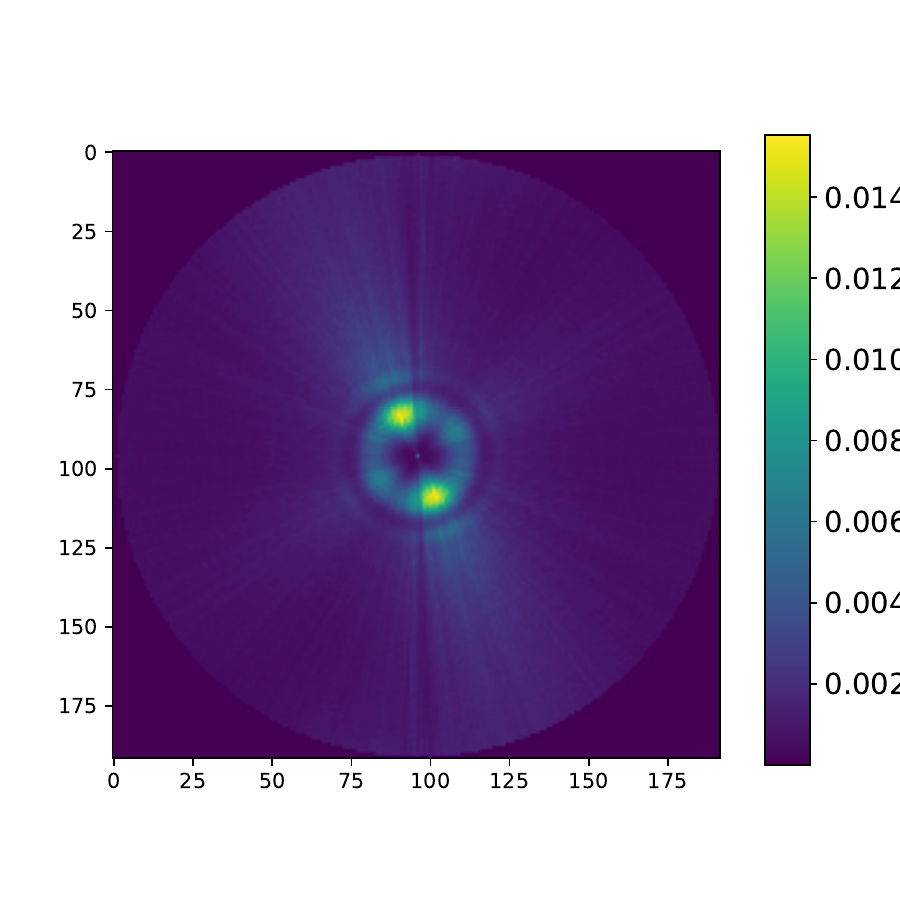}
        \caption{$z=0$ slice of $H$}
    \end{subfigure}
    \begin{subfigure}[b]{0.4\textwidth}
        \centering
        \includegraphics[width=\textwidth]{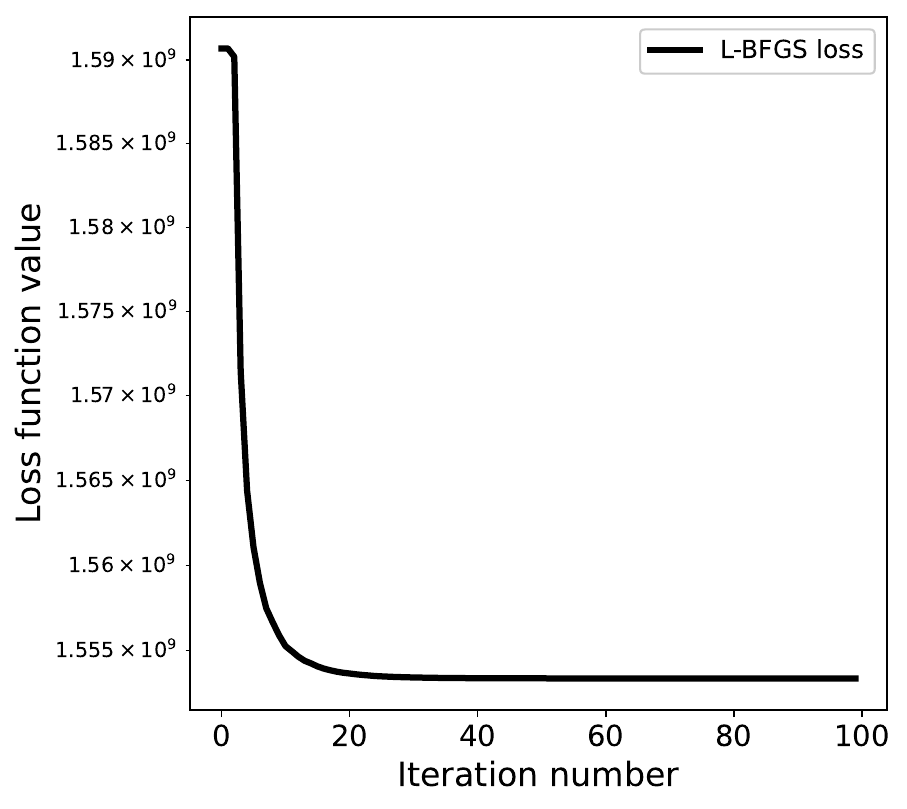}
        \caption{Convergence of L-BFGS}
    \end{subfigure}
    \caption{
        (a) The condition number $\kappa(H)$ for the experimental dataset used in 
        Section~\ref{sec:experiments}, with $N = 30000$ images, grid of dimensions
        $192\times 192 \times 192$, and $\lambda=10^{-8}$:
        lower bound of $\kappa(H)$ (blue, solid) computed using~\eqref{eq:cond number p},
        the value of $\kappa(H)$ for the given dataset without CTF (red, dashed) 
        and with CTF (green, dashdot).
        The condition numbers computed for this plot are for $H$ with nearest-neighbor
        interpolation.
        (b) The cross-section of the diagonal of the Hessian matrix $H$ computed in 
        (a) without CTF (red, dashed) and with CTF (green, dashdot).
        (c) The $z=0$ plane of the diagonal of the Hessian matrix $H$ for the given dataset, 
        computed using trilinear interpolation and reshaped as a 3D volume. Note the effect 
        of the CTF, in contrast to Figure~\ref{fig:cond number plots}(c), where the particle 
        images are not corrupted by the CTF.
        (d) The convergence of L-BFGS on problem~\eqref{eq:optim cp} with trilinear interpolation and the dataset described in Section~\ref{sec:experiments} to obtain the ground truth volume against which we evaluate the solutions of the SGD algorithms.
    }
    \label{fig:cond number plots real data}
\end{figure}

\begin{figure}%[h]
    \centering
    \begin{subfigure}[b]{0.4\textwidth}
        \centering
        \includegraphics[width=\textwidth]{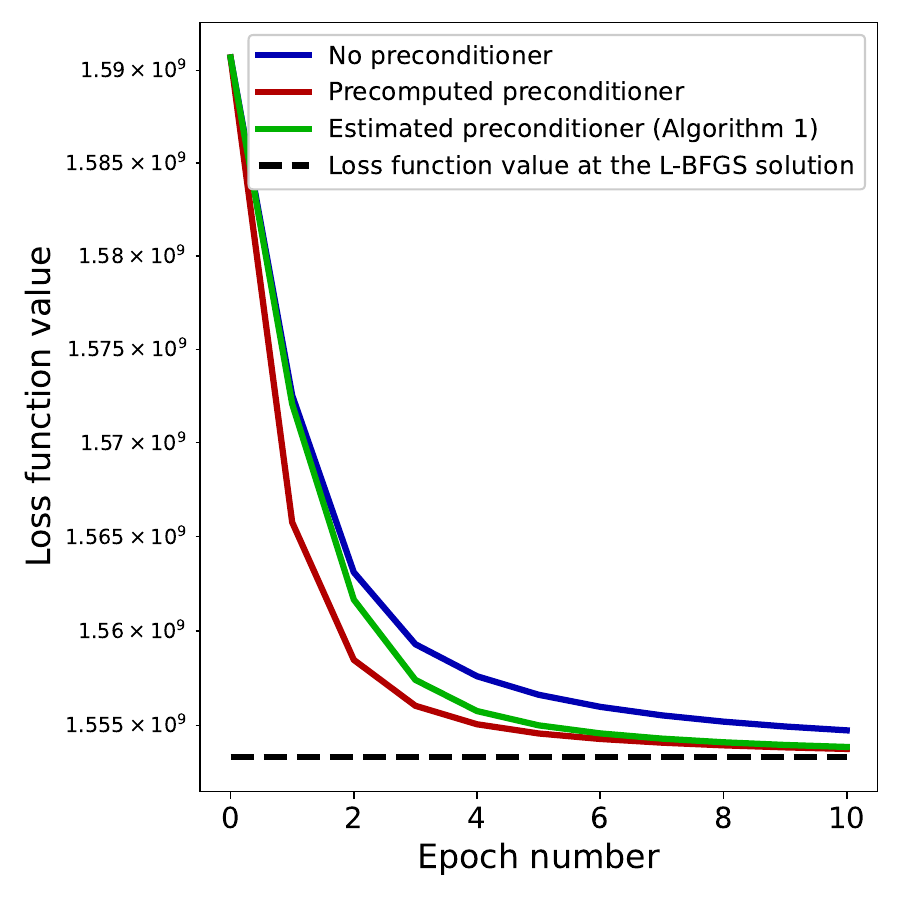}
        \caption{Loss}
        \label{fig:results loss}
    \end{subfigure}
    \begin{subfigure}[b]{0.4\textwidth}
        \centering
        \includegraphics[width=\textwidth]{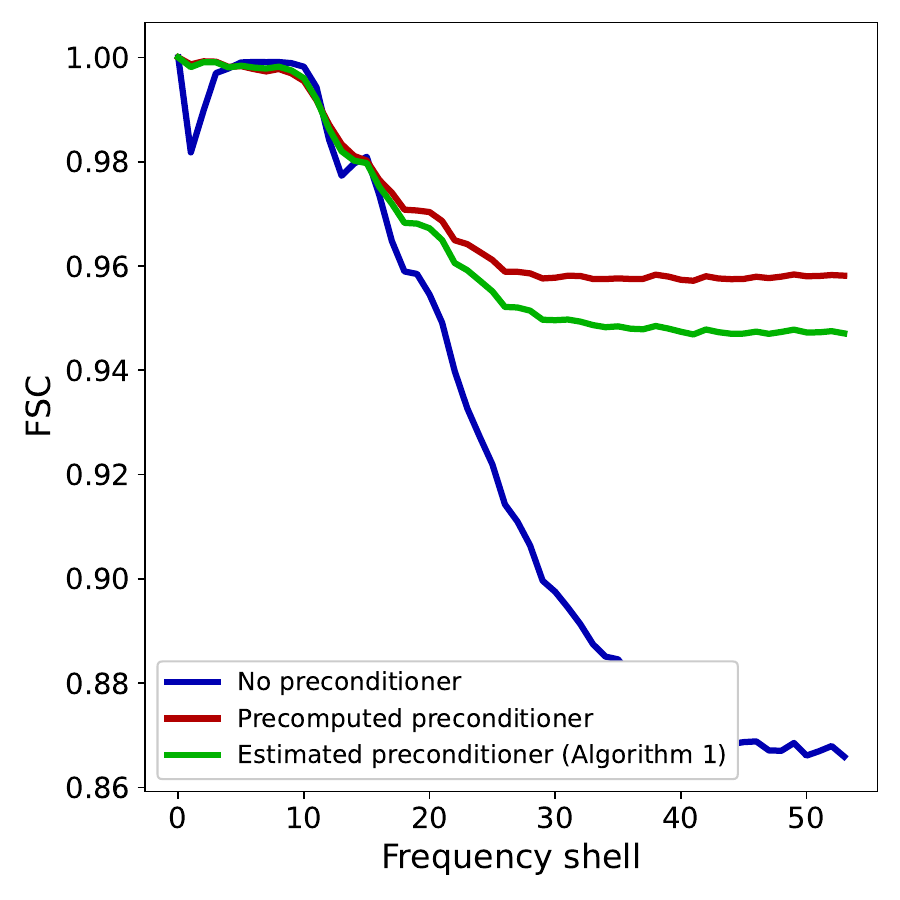}
        \caption{FSC with ground truth}
        \label{fig:results fsc}
    \end{subfigure}
    \caption{Results of the numerical experiments with trilinear interpolation 
        projection matrices $P_{\phi}^{tri}$ for $N=30000$ images of 
        $192 \times 192$ pixels. Panel (a) shows the loss function value and panel 
        (b) shows the FSC of the final reconstruction with a ground truth solution for SGD 
        with no preconditioner (blue), SGD with a precomputed preconditioner
        (red), and Algorithm~\ref{alg:precon sgd}, namely SGD with an estimated preconditioner (green). 
        The black dashed line in (a) shows the loss function value of the final L-BFGS solution used as ground truth, against which we compute the FSC in (b).
    }
    \label{fig:results}
\end{figure}

\begin{figure}%[h]
    \centering
    \begin{subfigure}[b]{0.32\textwidth}
        \centering
        \includegraphics[width=\textwidth]{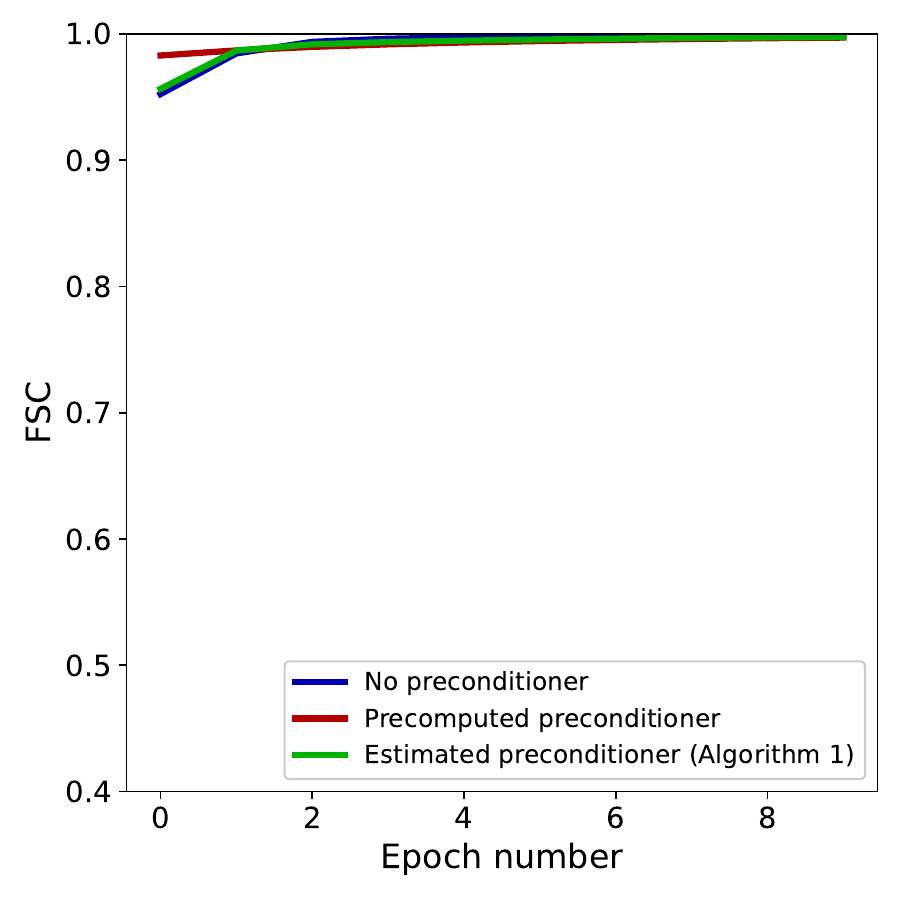}
        \caption{Fourier shell 10/55}
        \label{fig:fsc epoch low}
    \end{subfigure}
    \hfill
    \begin{subfigure}[b]{0.32\textwidth}
        \centering
        \includegraphics[width=\textwidth]{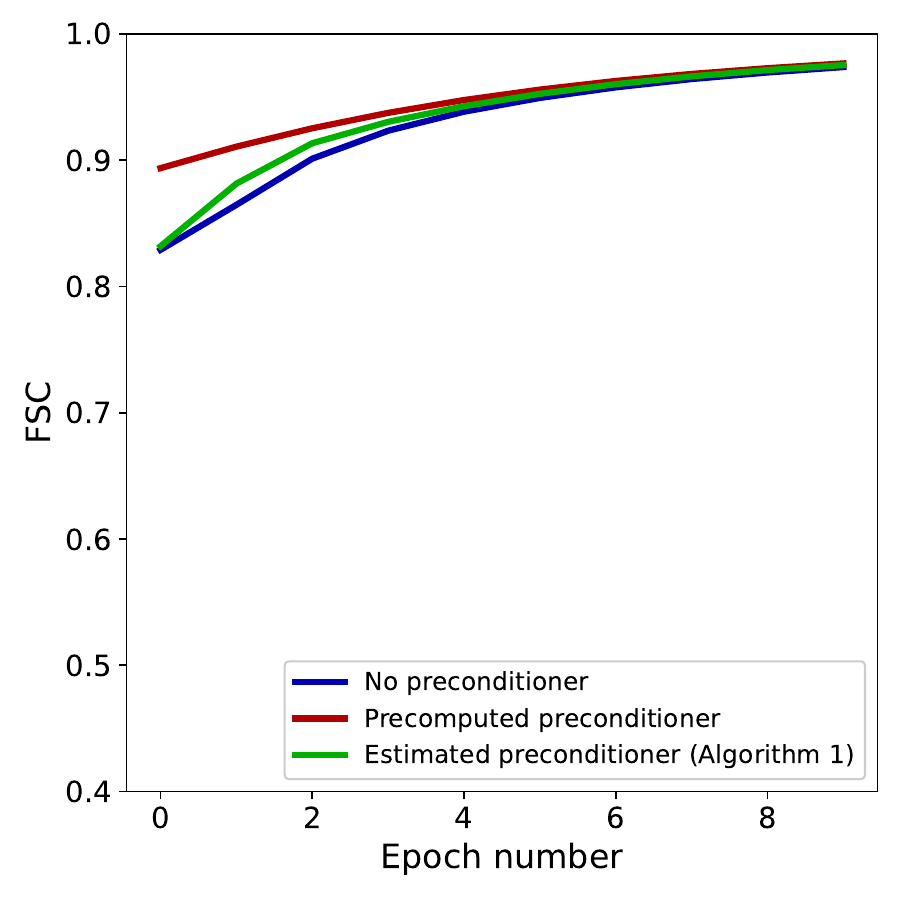}
        \caption{Fourier shell 17/55}
        \label{fig:fsc epoch medium}
    \end{subfigure}
    \hfill
    \begin{subfigure}[b]{0.32\textwidth}
        \centering
        \includegraphics[width=\textwidth]{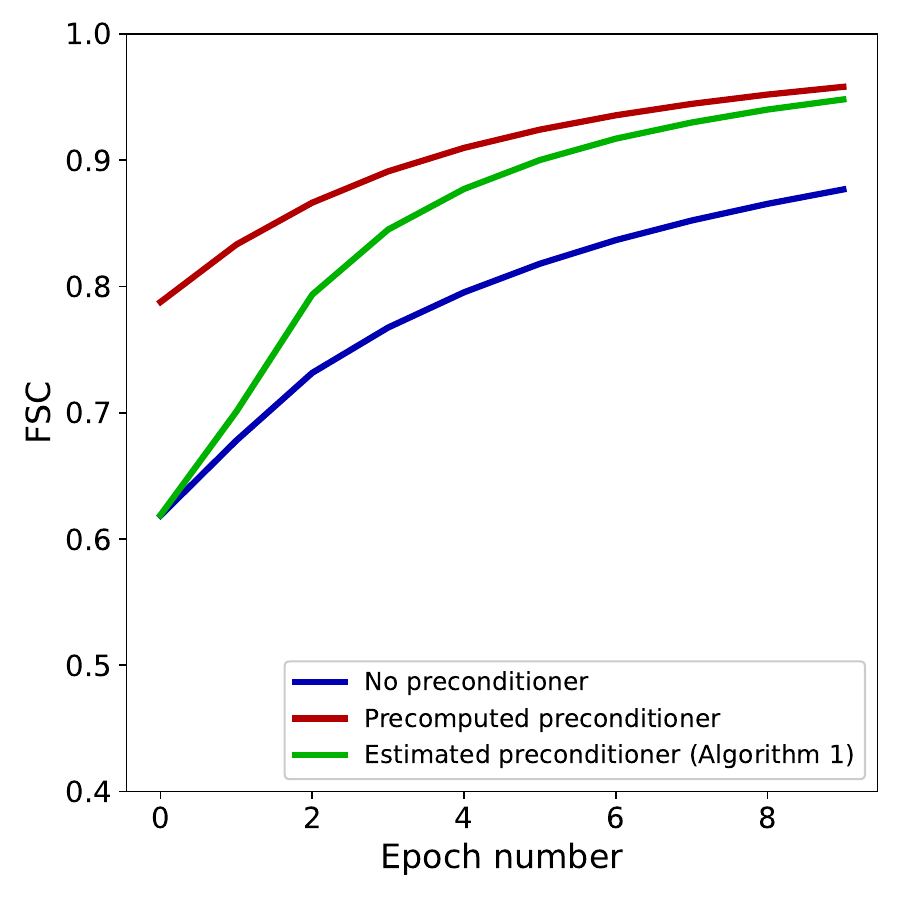}
        \caption{Fourier shell 40/55}
        \label{fig:fsc epoch high}
    \end{subfigure}
    \caption{FSC for individual Fourier shells across epochs for a low-frequency
        shell (a), a middle-frequency shell (b) and a high-frequency shell (c).
        While preconditioning is not required at low resolution, as seen in (a), 
        its effect becomes increasingly apparent at higher resolution. Note that, 
        since in these experiments the projection matrices use trilinear interpolation,
        the Hessian matrix $H$ is not diagonal, and yet the diagonal preconditioner
        is effective. 
    } 
    \label{fig:fsc epoch}
\end{figure}

\begin{figure}%[h]
    \centering
    \begin{subfigure}[b]{0.3285\textwidth}
        \centering
        \includegraphics[width=\textwidth]{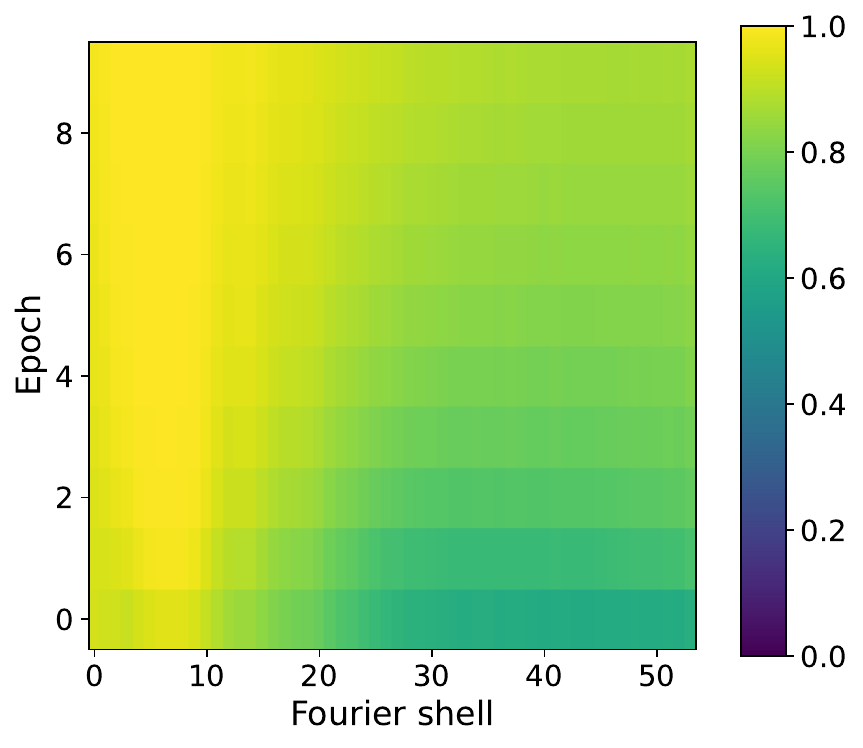}
        \caption{No preconditioner}
        \label{fig:fsc shell epoch sgd}
    \end{subfigure}
    \hfill
    \begin{subfigure}[b]{0.3285\textwidth}
        \centering
        \includegraphics[width=\textwidth]{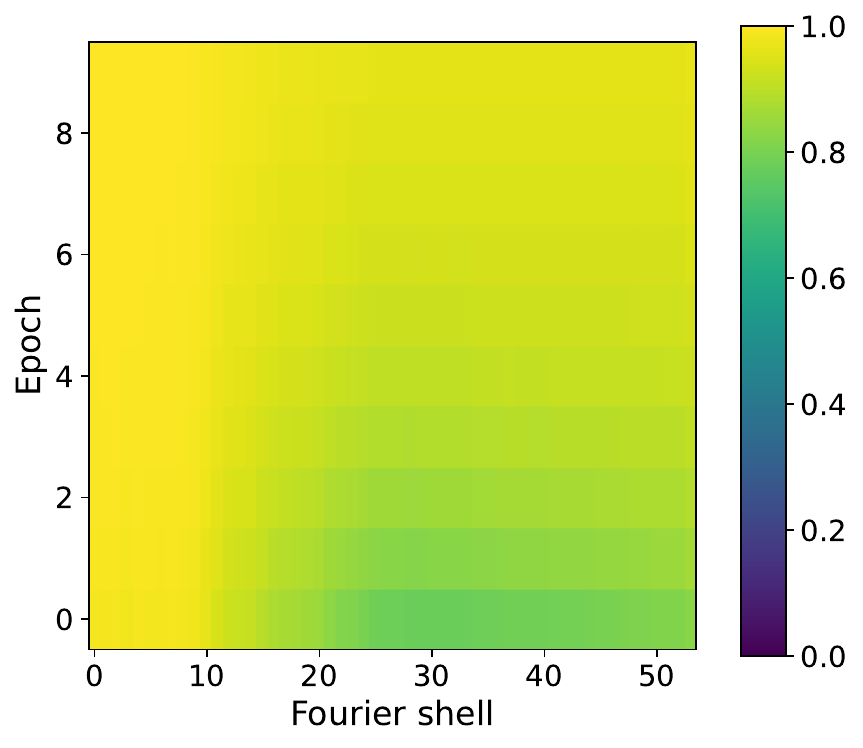}
        \caption{Precomputed preconditioner} 
        \label{fig:fsc shell epoch sgd p}
    \end{subfigure}
    \hfill
    \begin{subfigure}[b]{0.3285\textwidth}
        \centering
        \includegraphics[width=\textwidth]{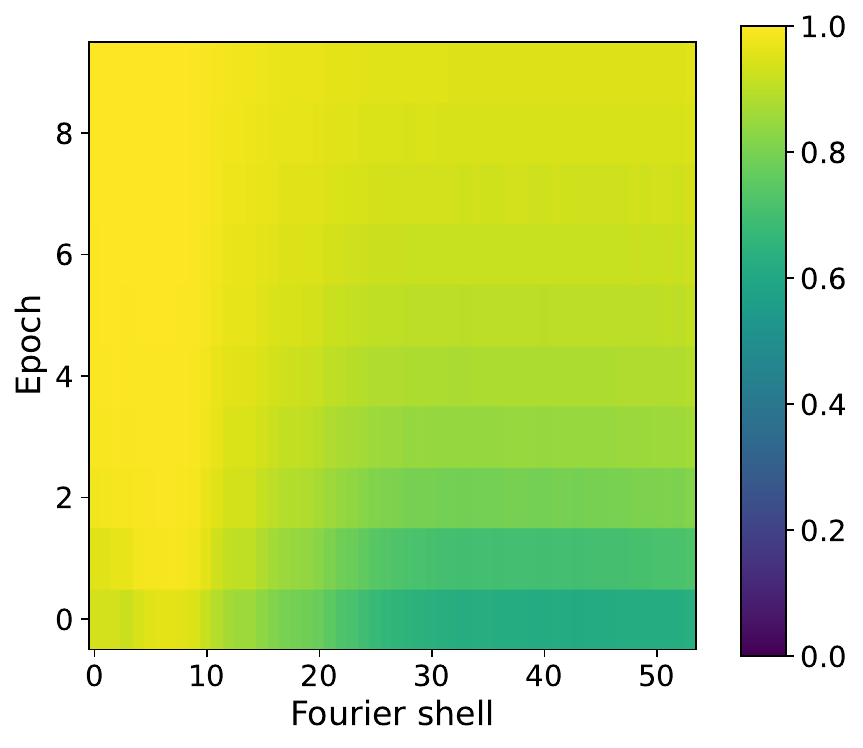}
        \caption{Algorithm~\ref{alg:precon sgd}}
        \label{fig:fsc shell epoch sgd o}
    \end{subfigure}
    \caption{FSC for individual Fourier shells as a function of the Fourier shell 
        index and epoch number for SGD with no preconditioner (a), SGD with a 
        precomputed preconditioner (b) and SGD with an estimated preconditioner (c), 
        where the projection operation is implemented using trilinear interpolation.} 
    \label{fig:fsc shell epoch}
\end{figure}

\begin{figure}
    \centering
    \begin{tabular}{cccc}
        \makecell{Ground truth \\(L-BFGS)} & 
        \makecell{No preconditioner} & 
        \makecell{Precomputed \\ preconditioner} & 
        \makecell{Algorithm~\ref{alg:precon sgd}} \\
        \hline

        \includegraphics[width=0.2\textwidth]{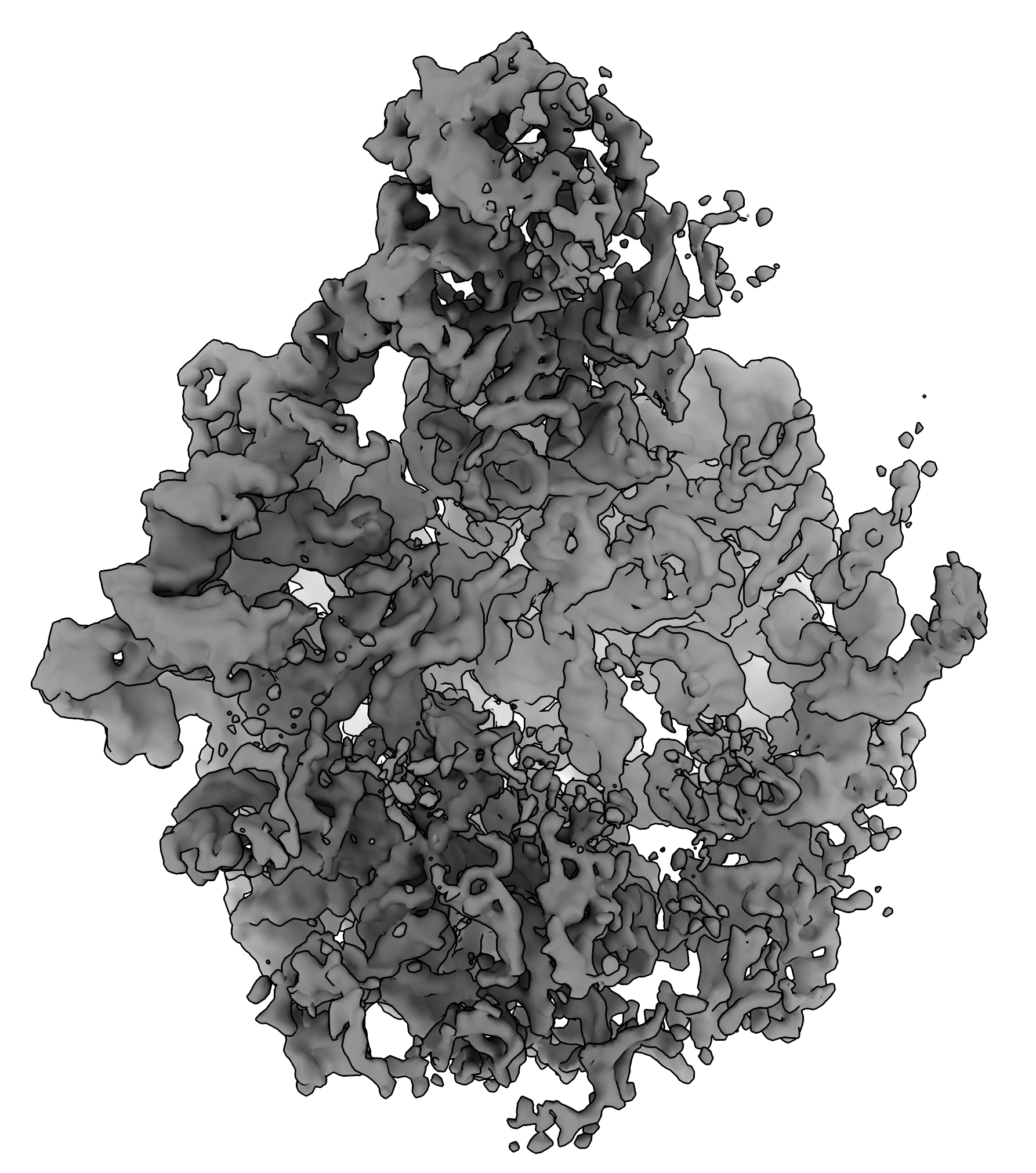} &
        \includegraphics[width=0.2\textwidth]{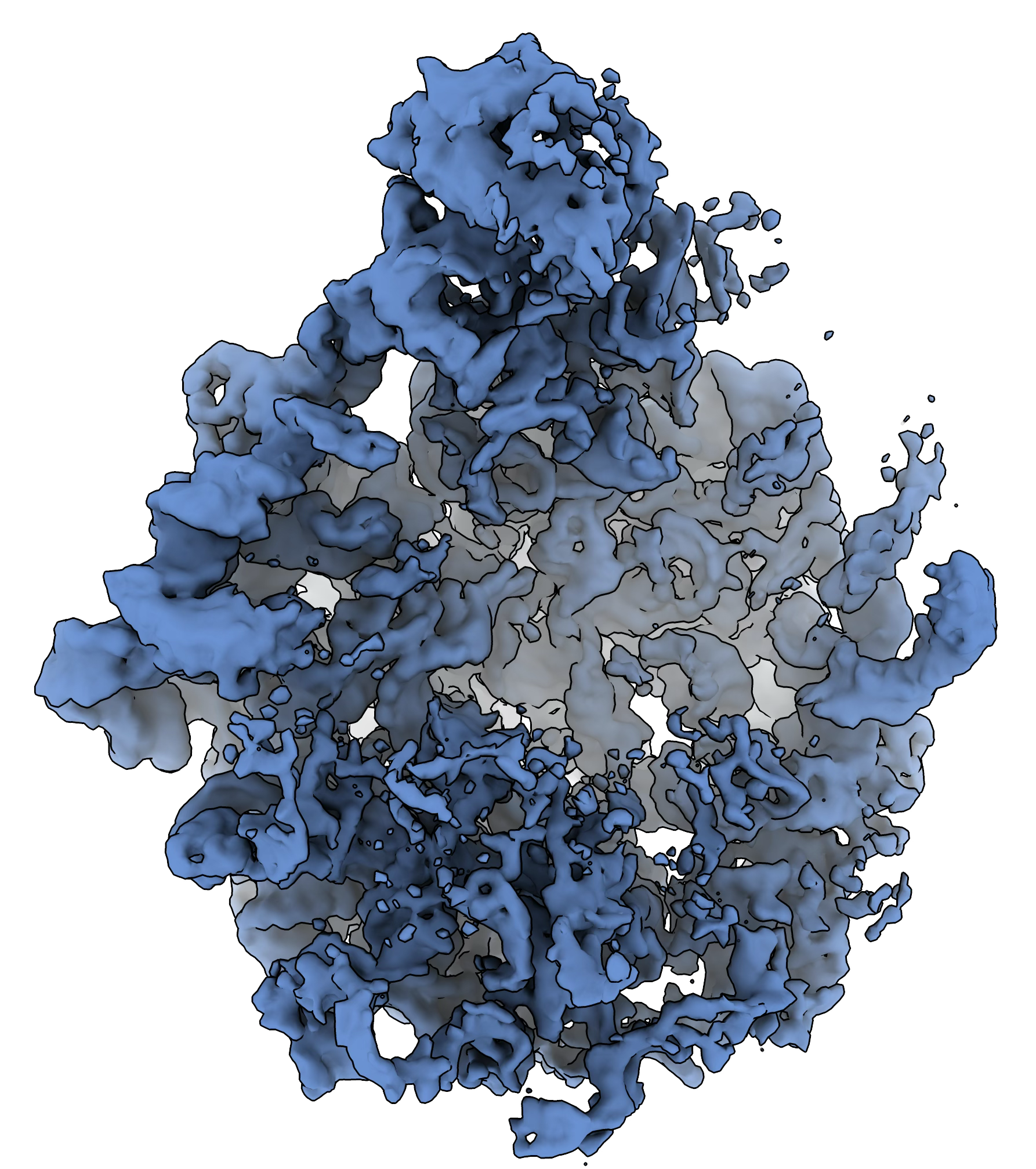} &
        \includegraphics[width=0.2\textwidth]{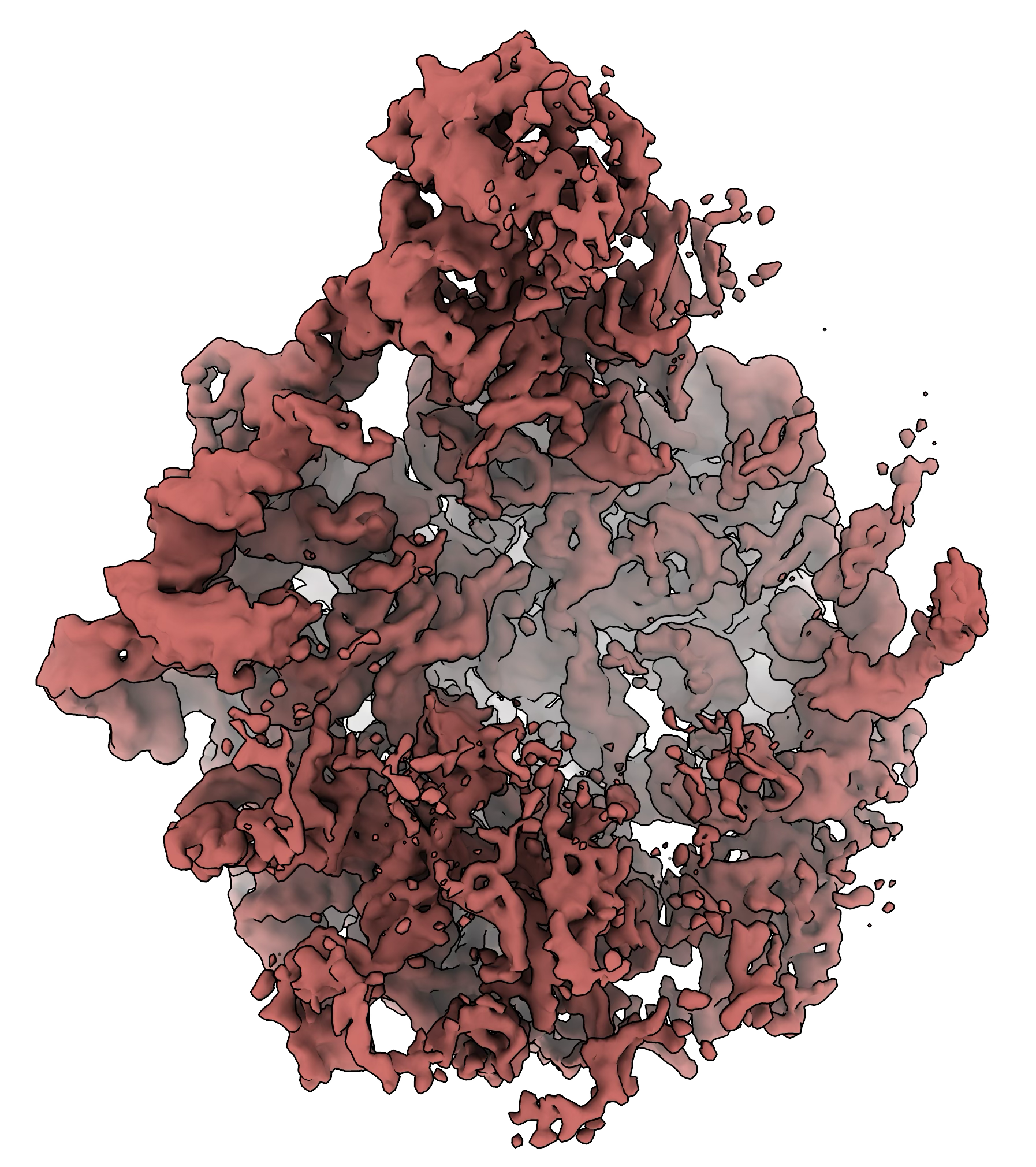} &
        \includegraphics[width=0.2\textwidth]{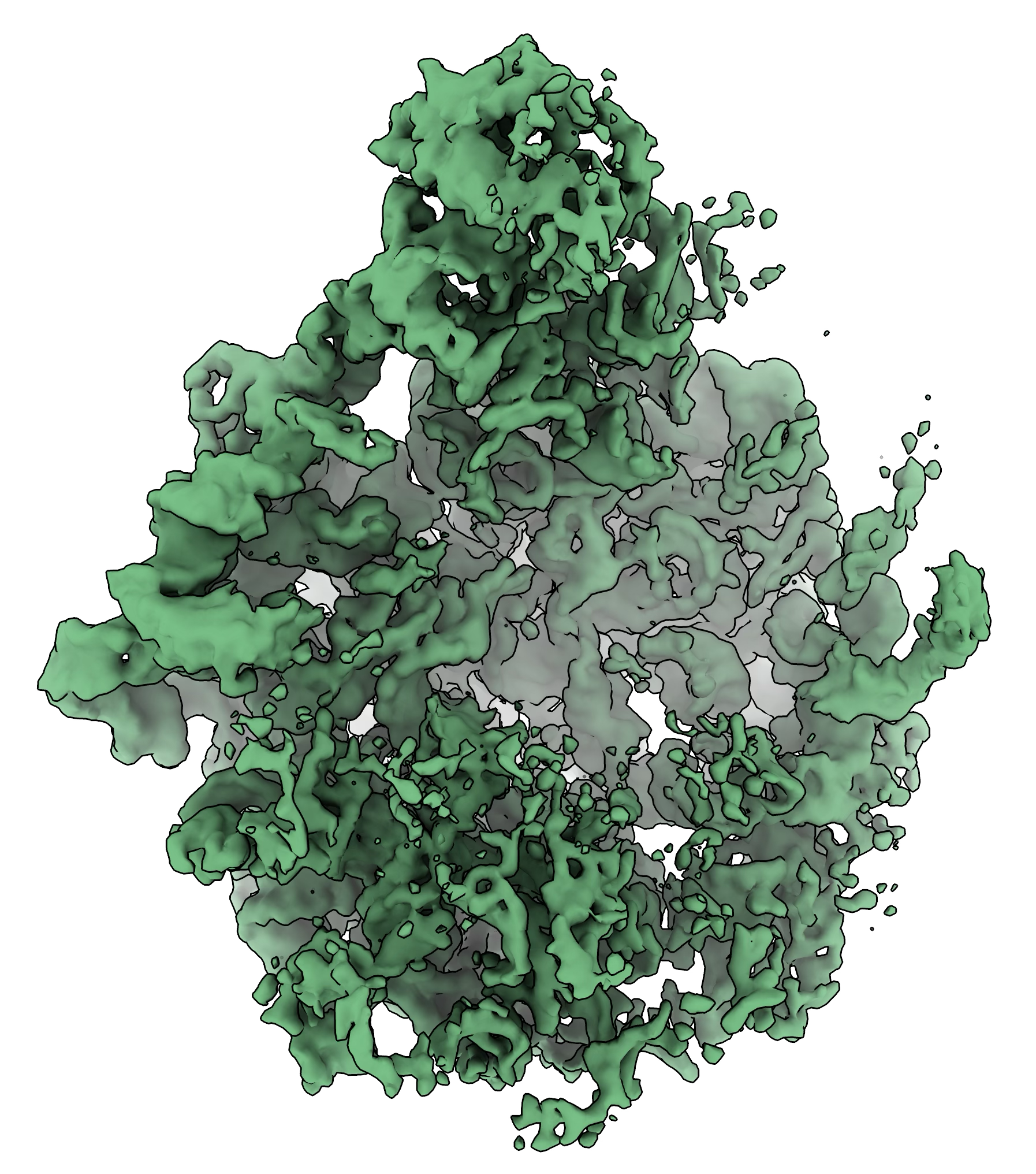} \\

        \includegraphics[width=0.2\textwidth]{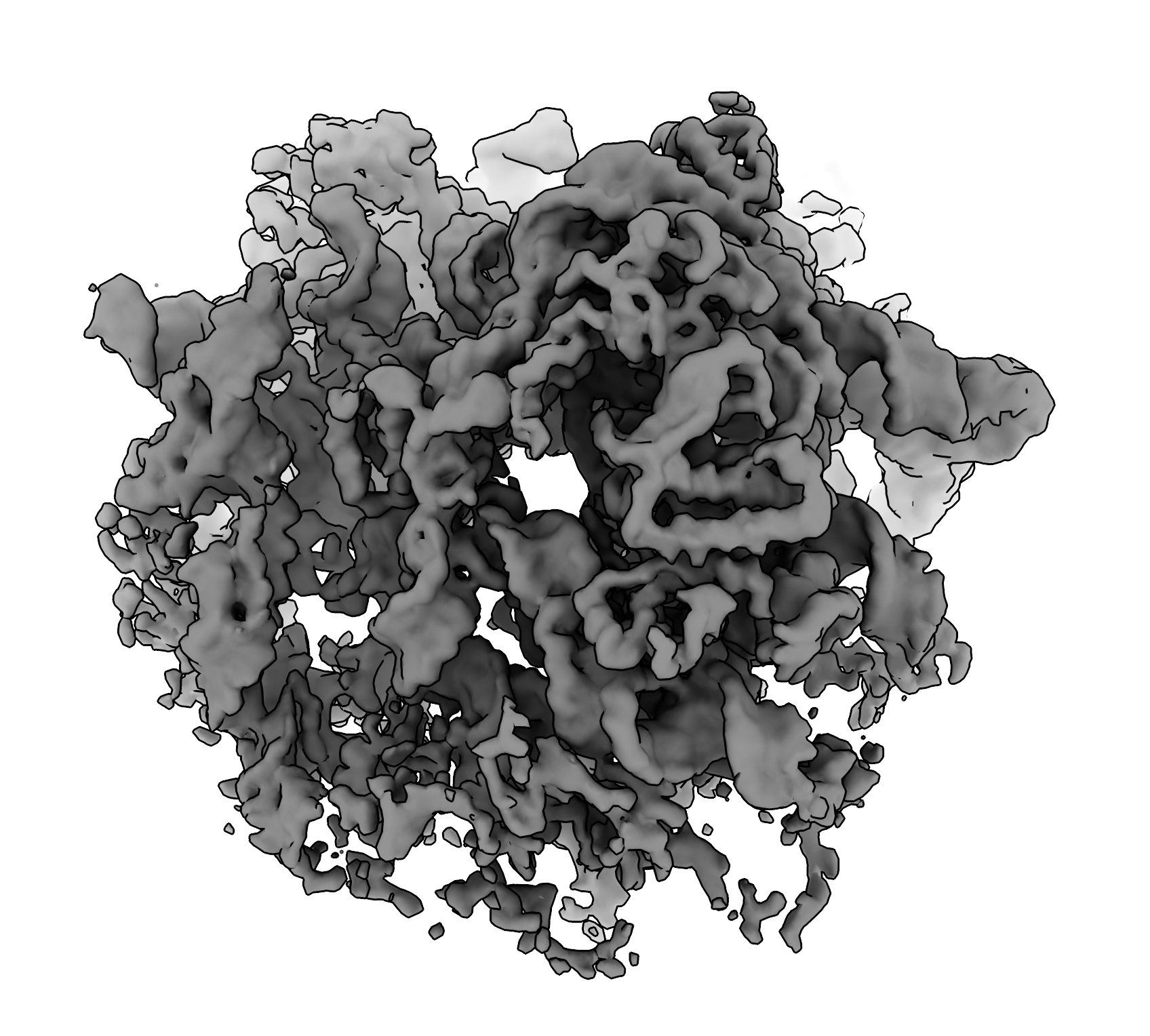} &
        \includegraphics[width=0.2\textwidth]{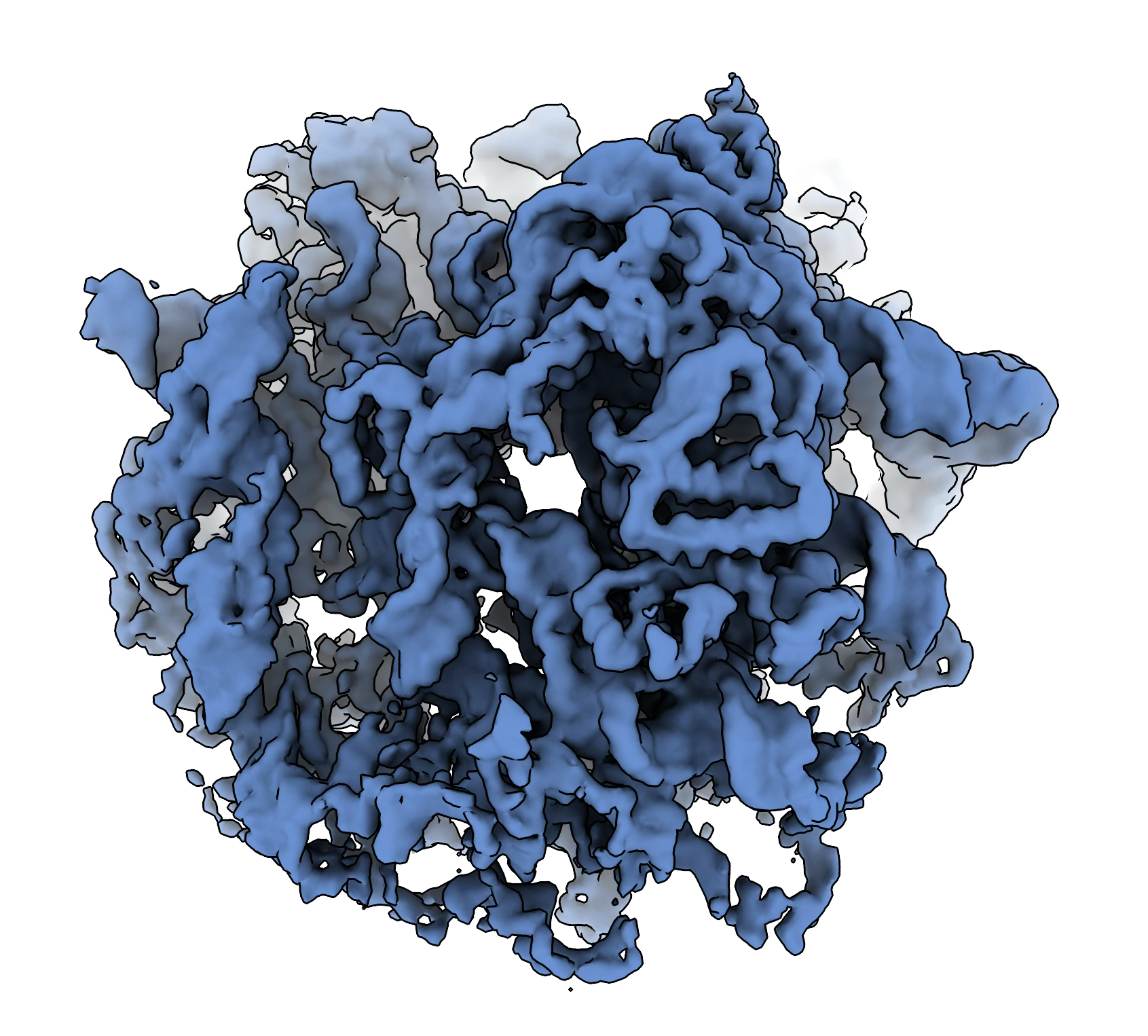} &
        \includegraphics[width=0.2\textwidth]{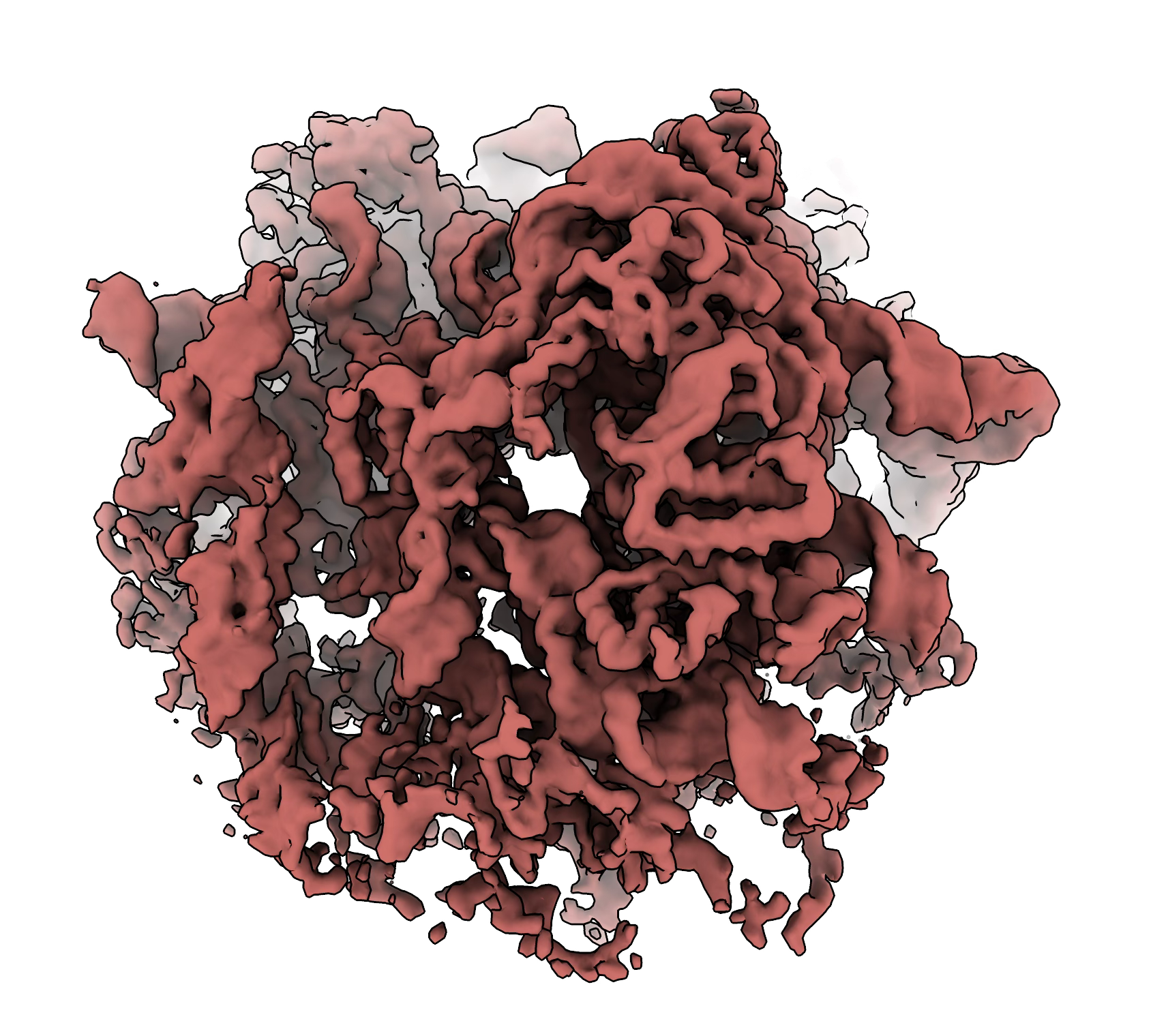} &
        \includegraphics[width=0.2\textwidth]{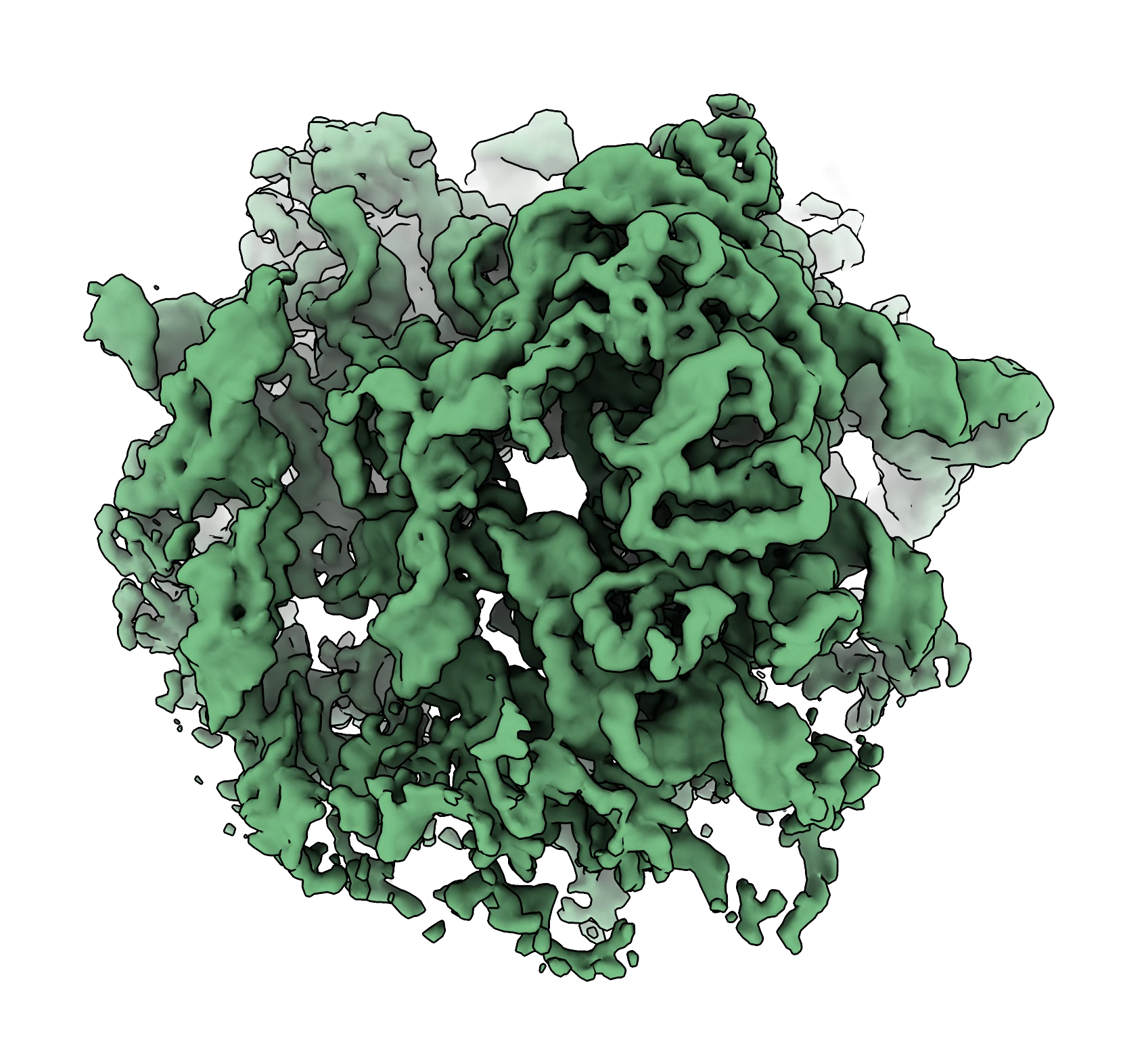} \\

        \includegraphics[width=0.2\textwidth]{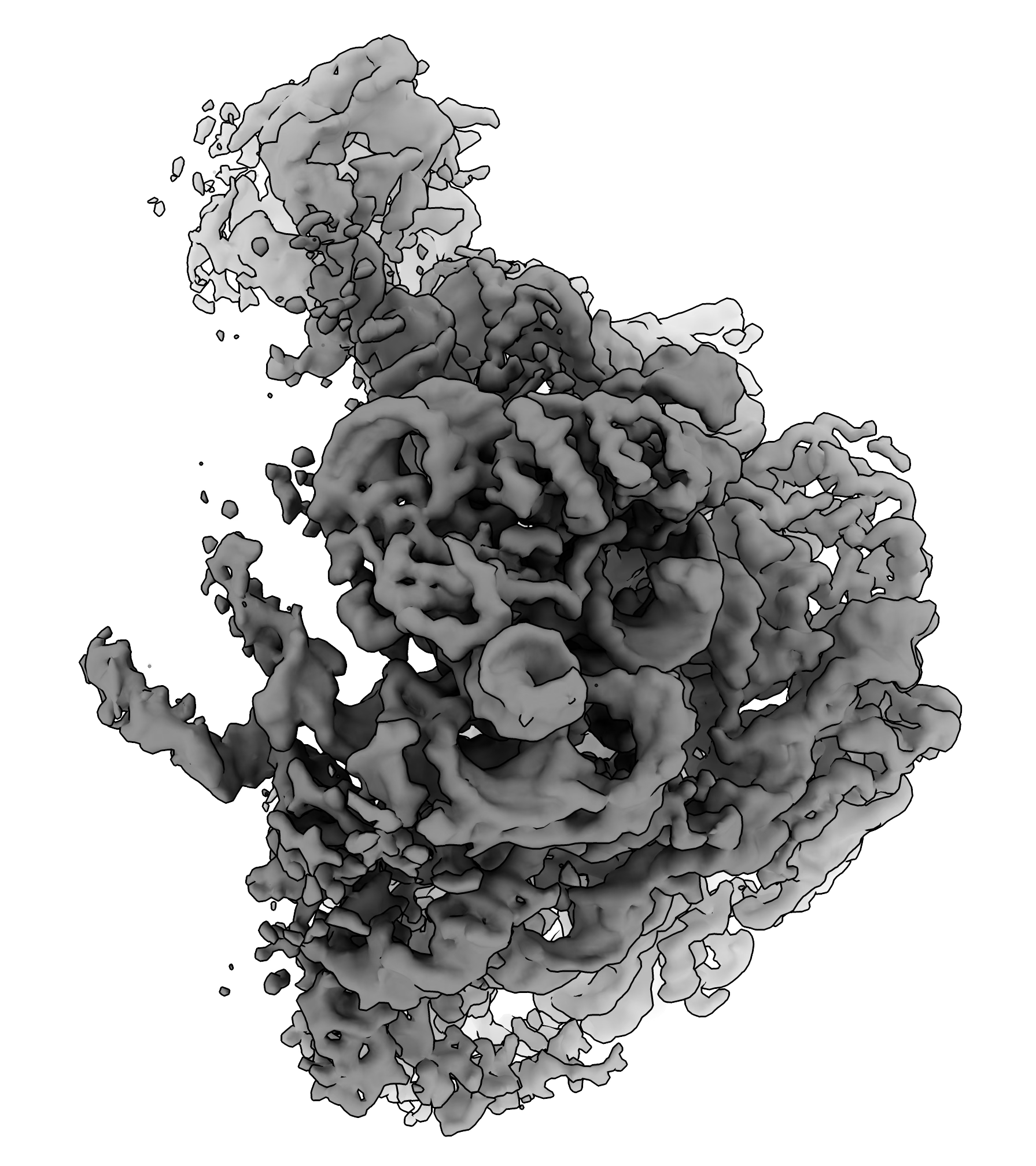} &
        \includegraphics[width=0.2\textwidth]{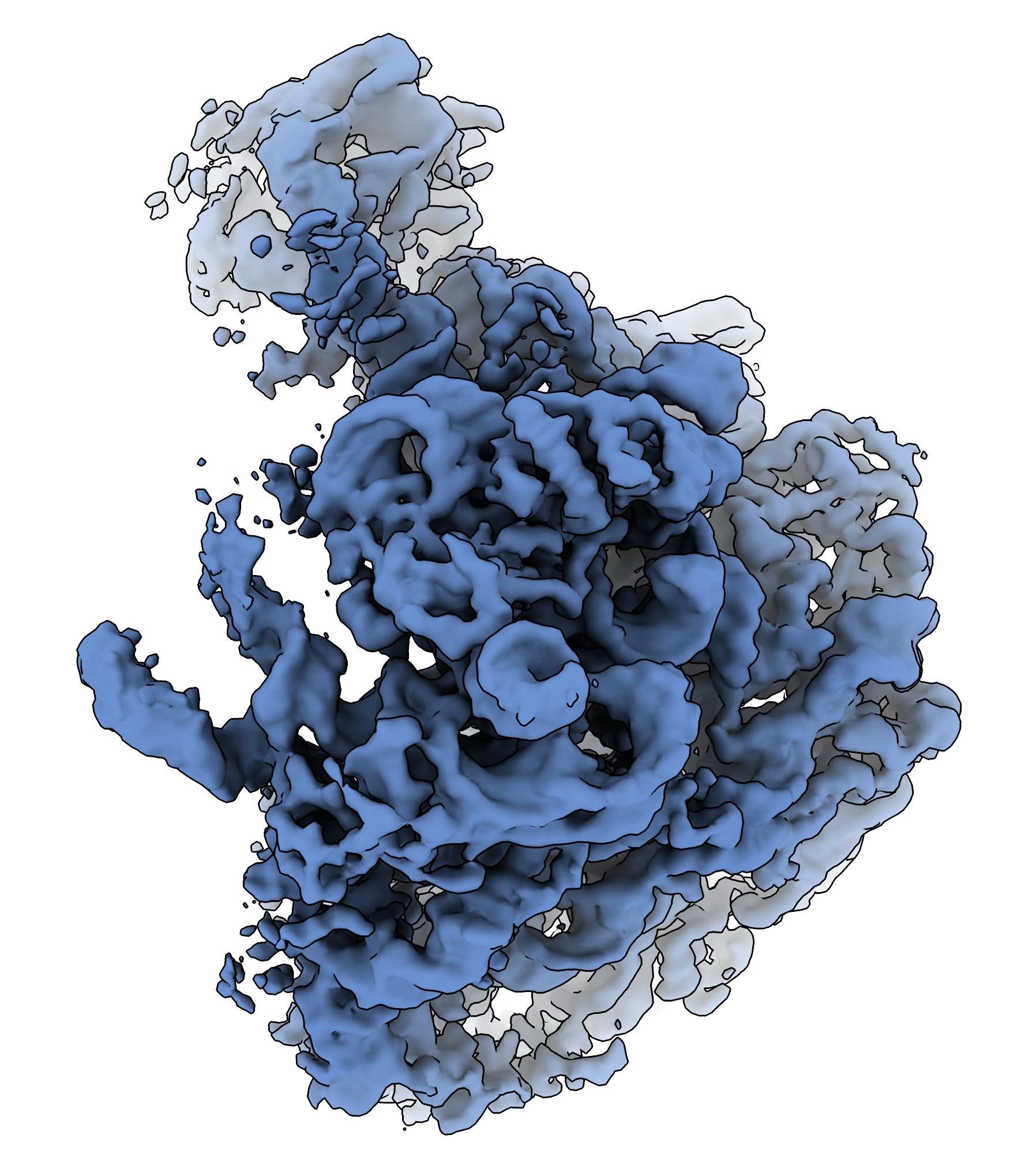} &
        \includegraphics[width=0.2\textwidth]{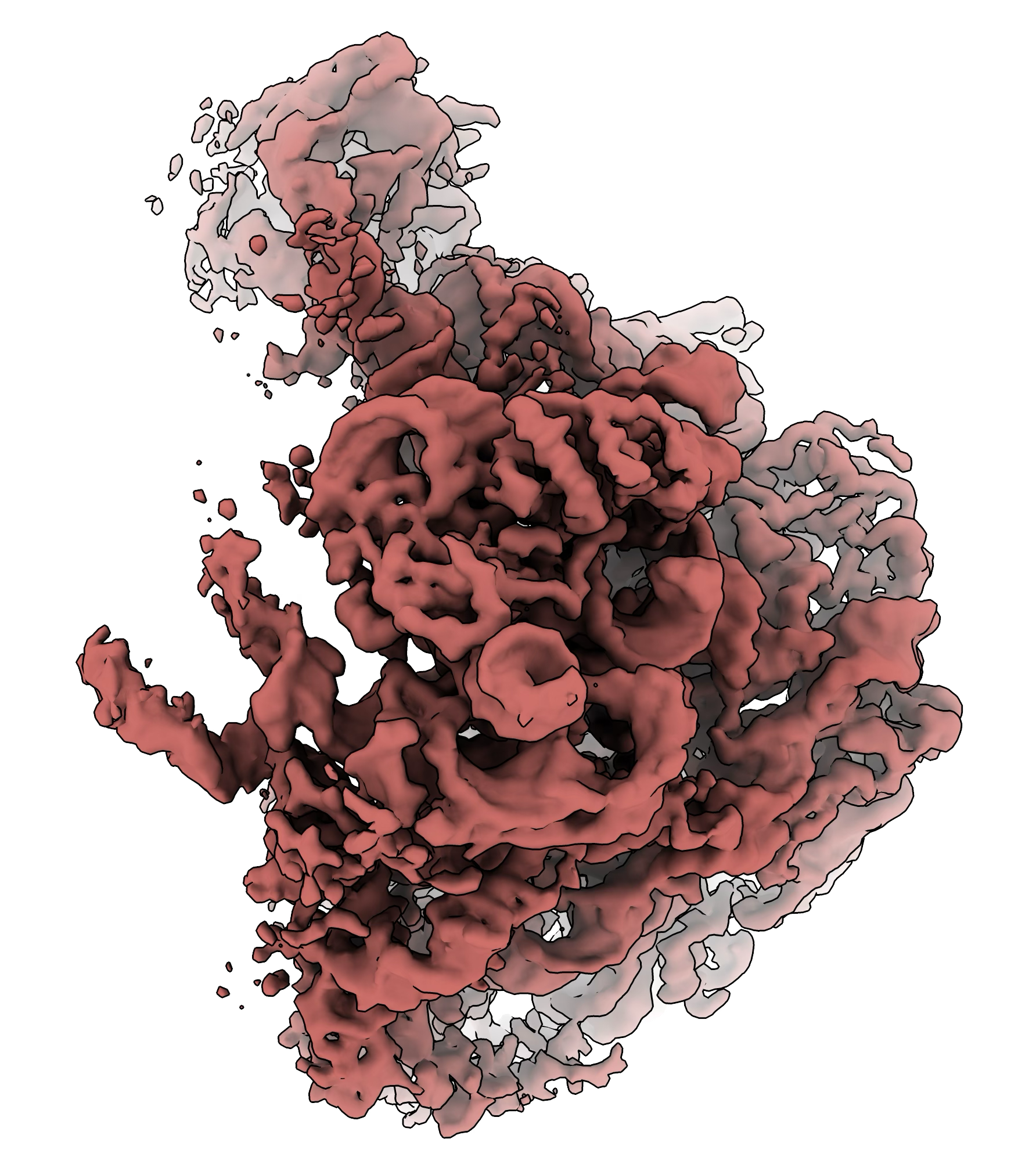} &
        \includegraphics[width=0.2\textwidth]{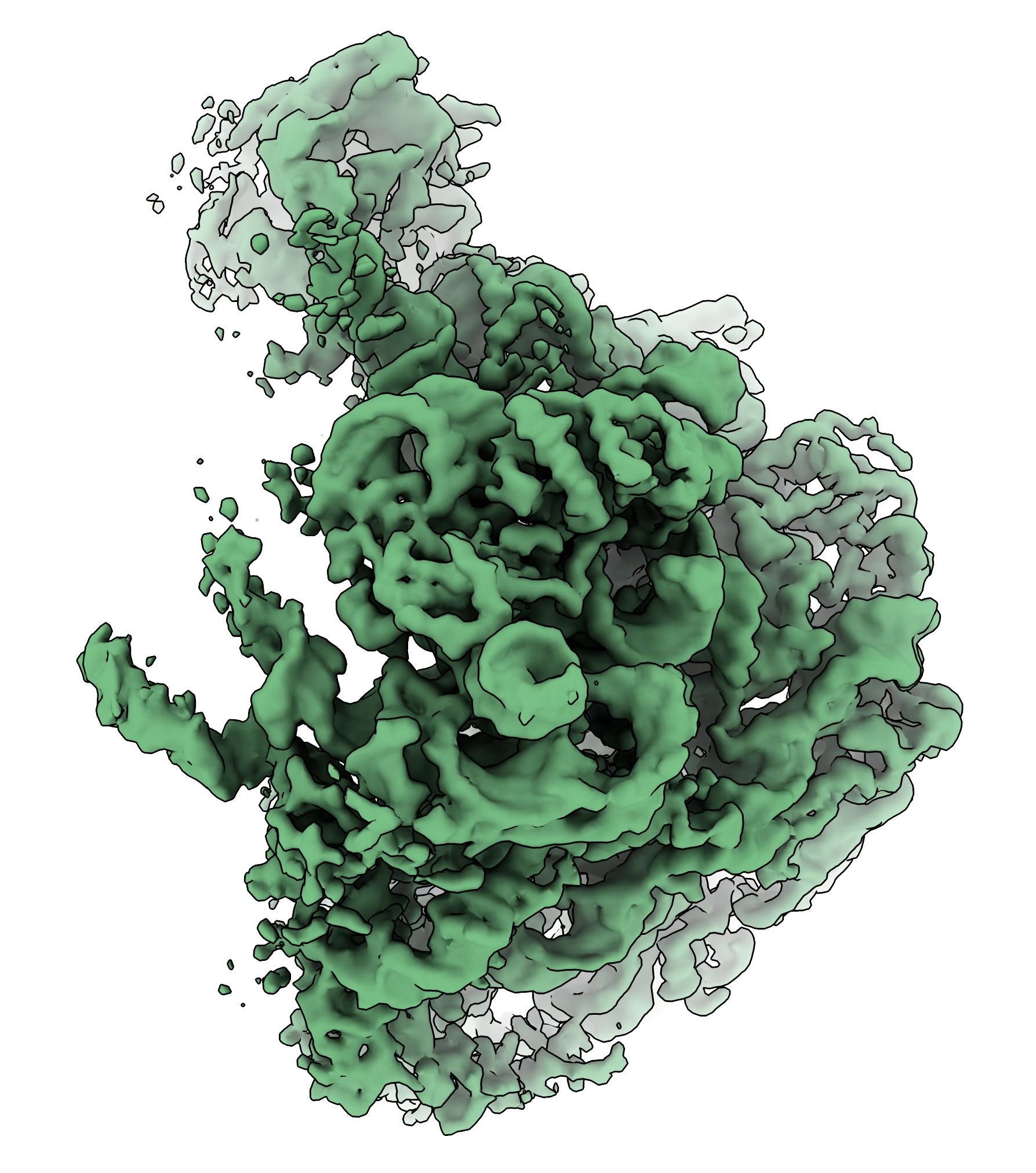} \\

        \includegraphics[width=0.2\textwidth]{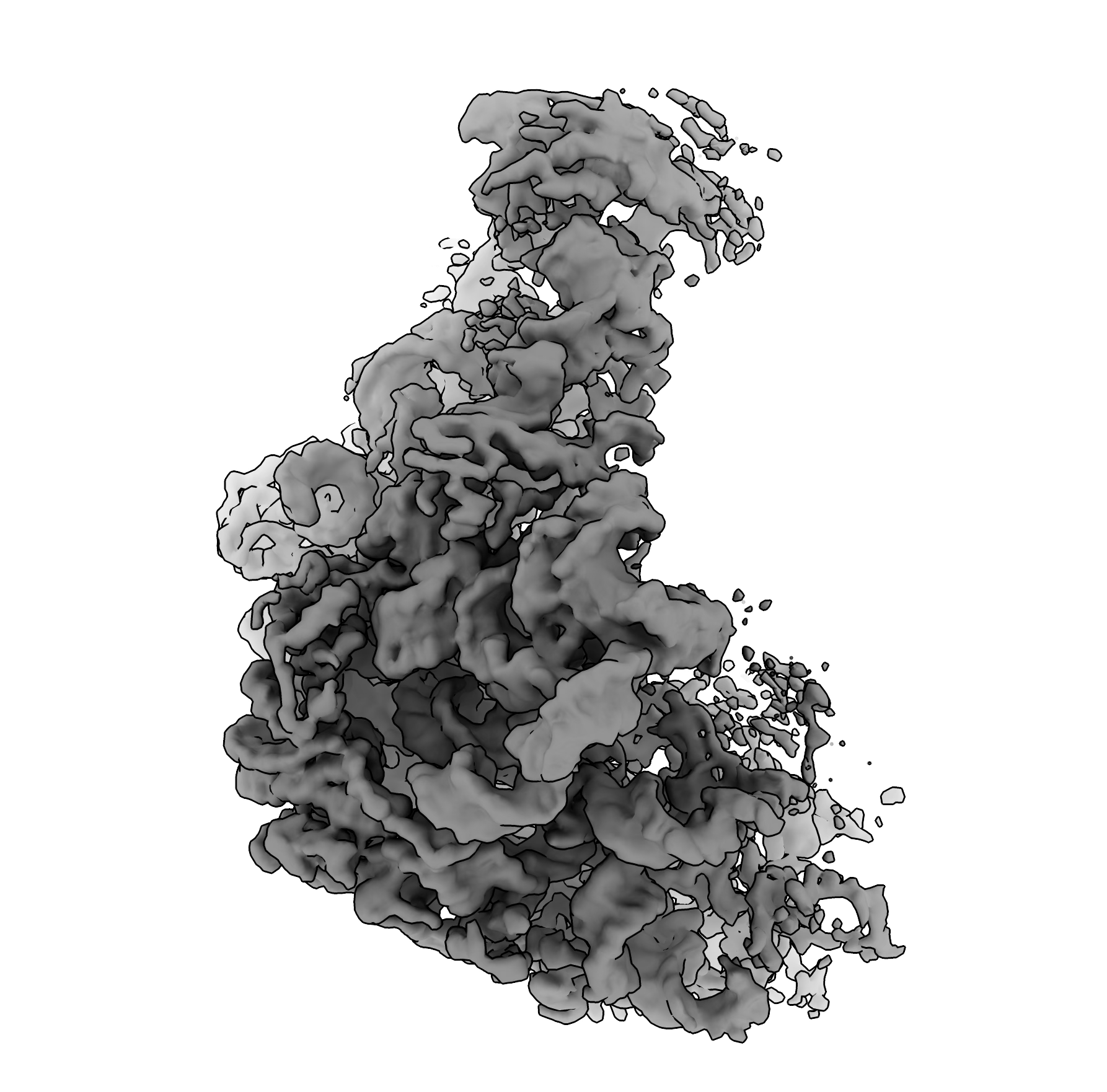} &
        \includegraphics[width=0.2\textwidth]{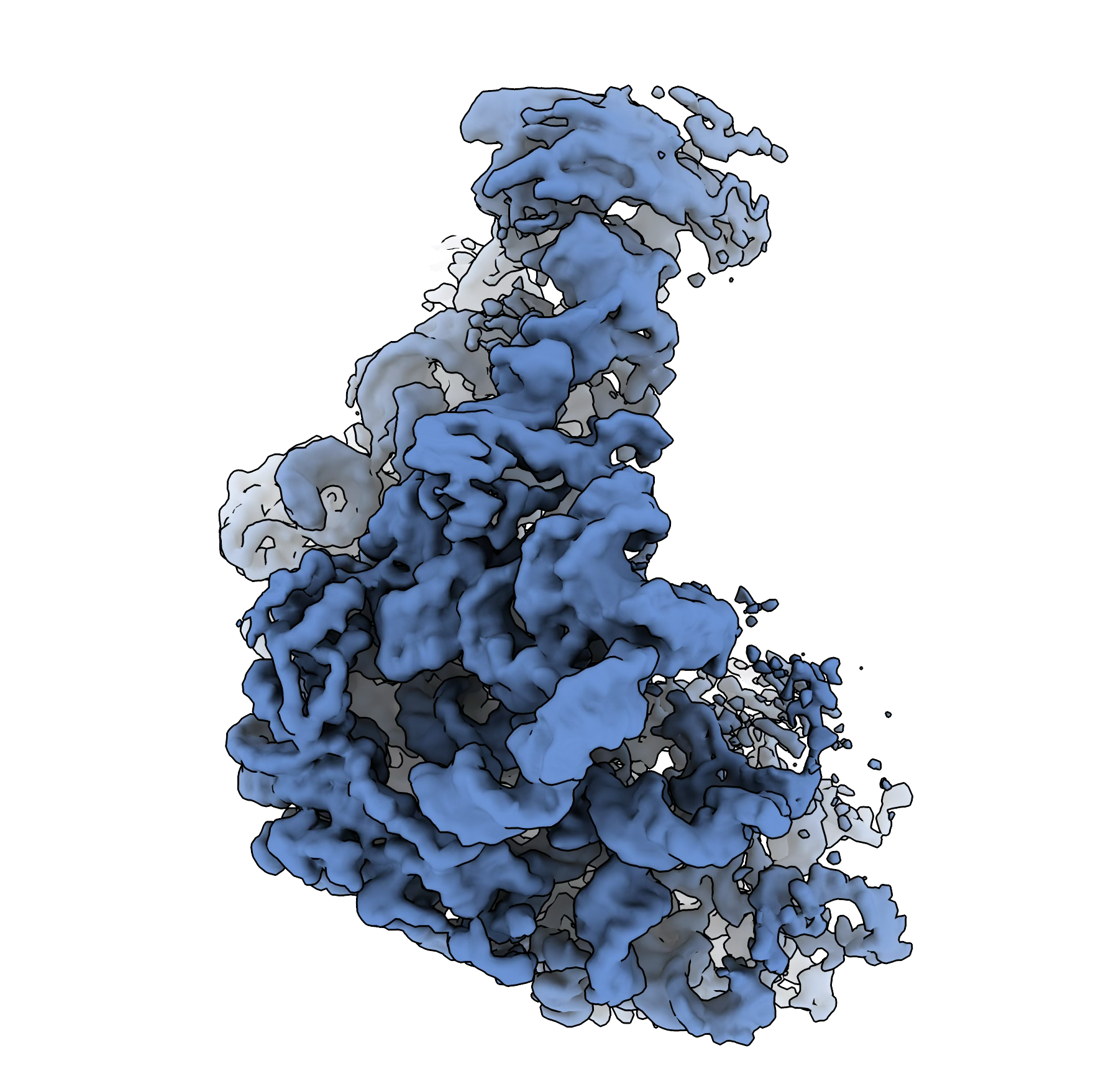} &
        \includegraphics[width=0.2\textwidth]{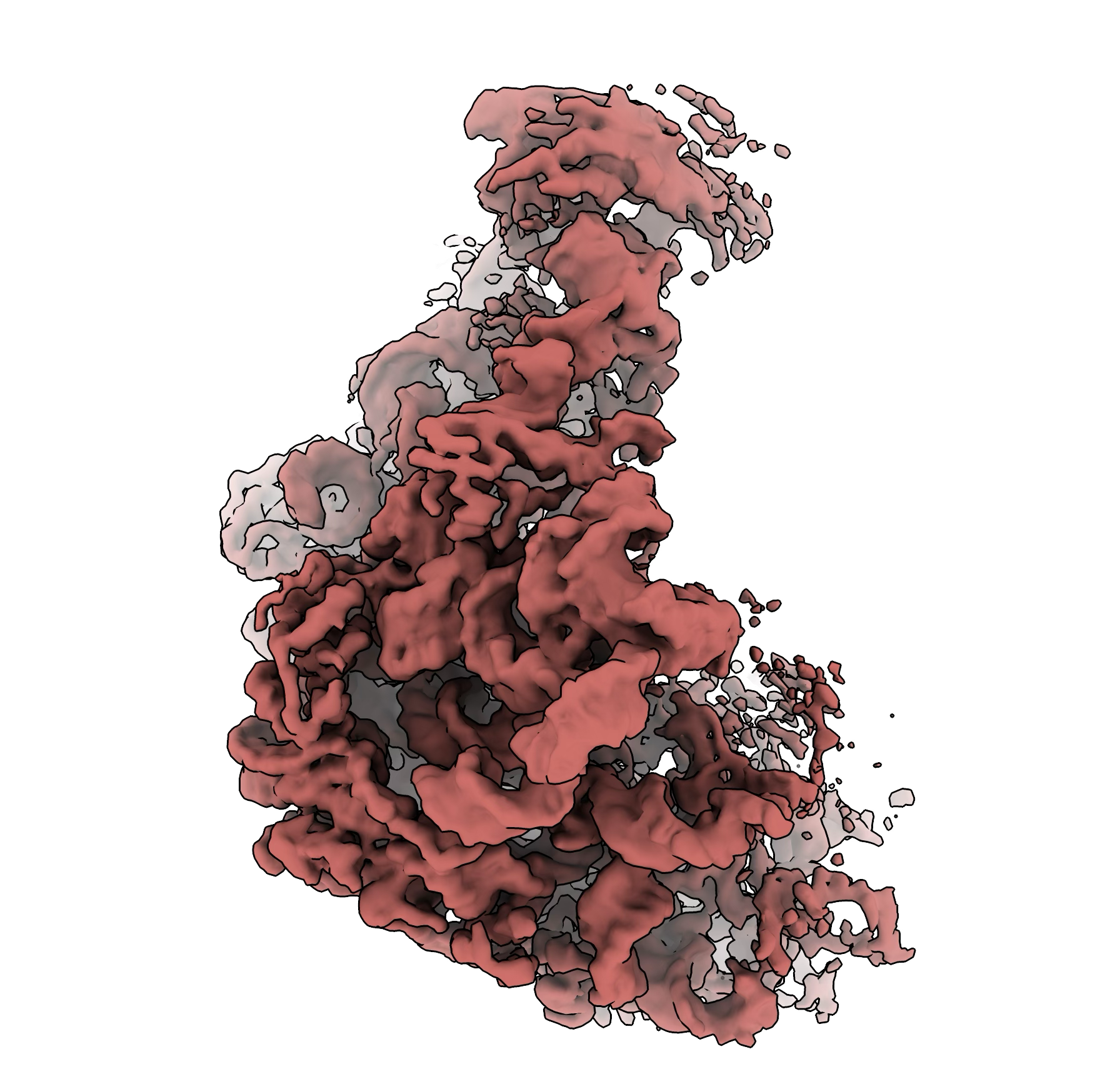} &
        \includegraphics[width=0.2\textwidth]{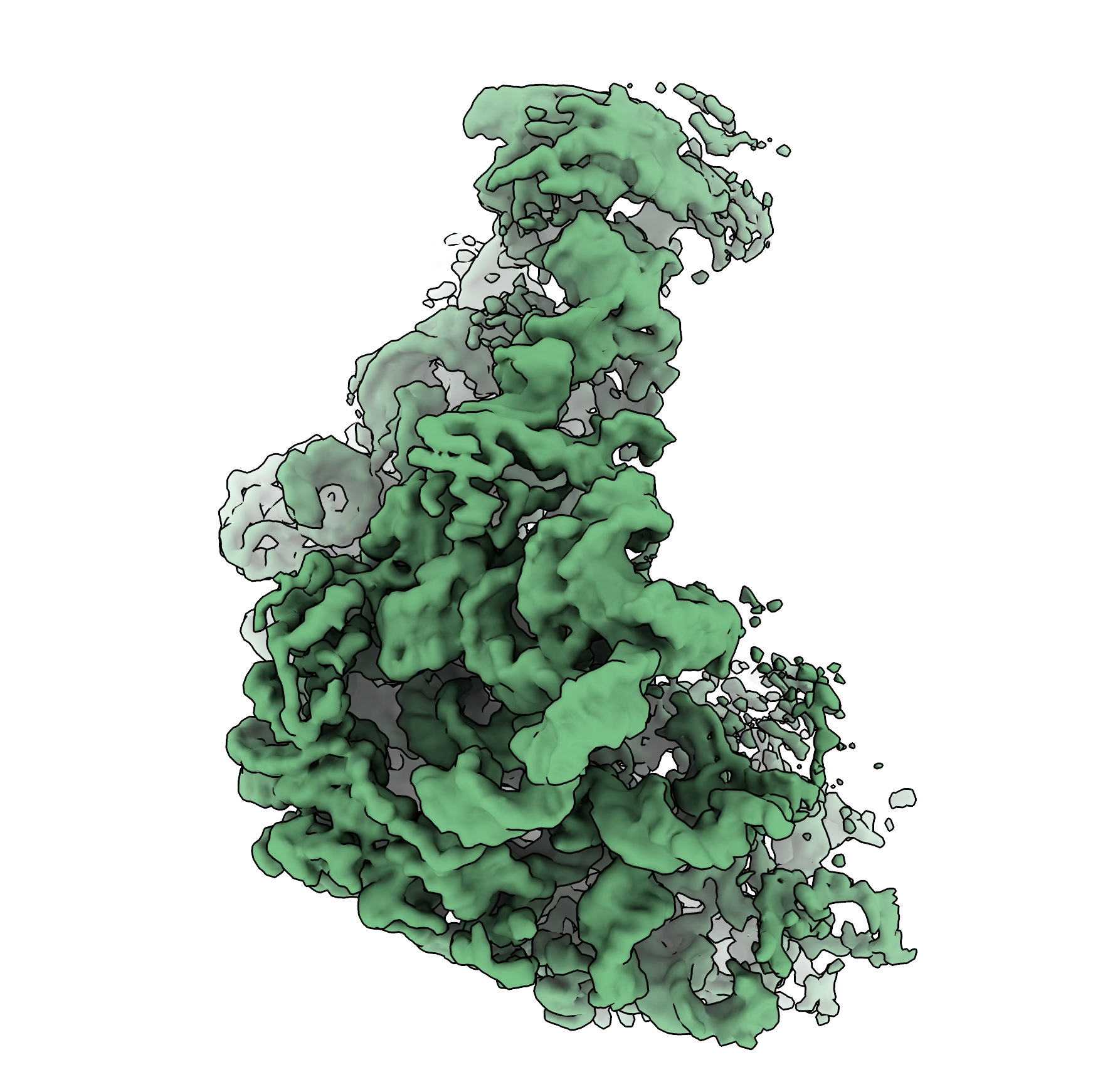} \\
    \end{tabular}
    \caption{A qualitative comparison of the reconstructions from the three SGD
        algorithms evaluated in Section~\ref{sec:experiments} (columns 2-4) and the ground truth 
        solution obtained using L-BFGS (first column), where each row shows the same view for all four volumes.
        Each volume is displayed as the isosurface at six standard deviations above its mean voxel value.
    }
    \label{fig:maps}
\end{figure}

\begin{figure}%[h]
    \centering
    \begin{subfigure}[b]{0.32\textwidth}
        \centering
        \includegraphics[width=\textwidth]{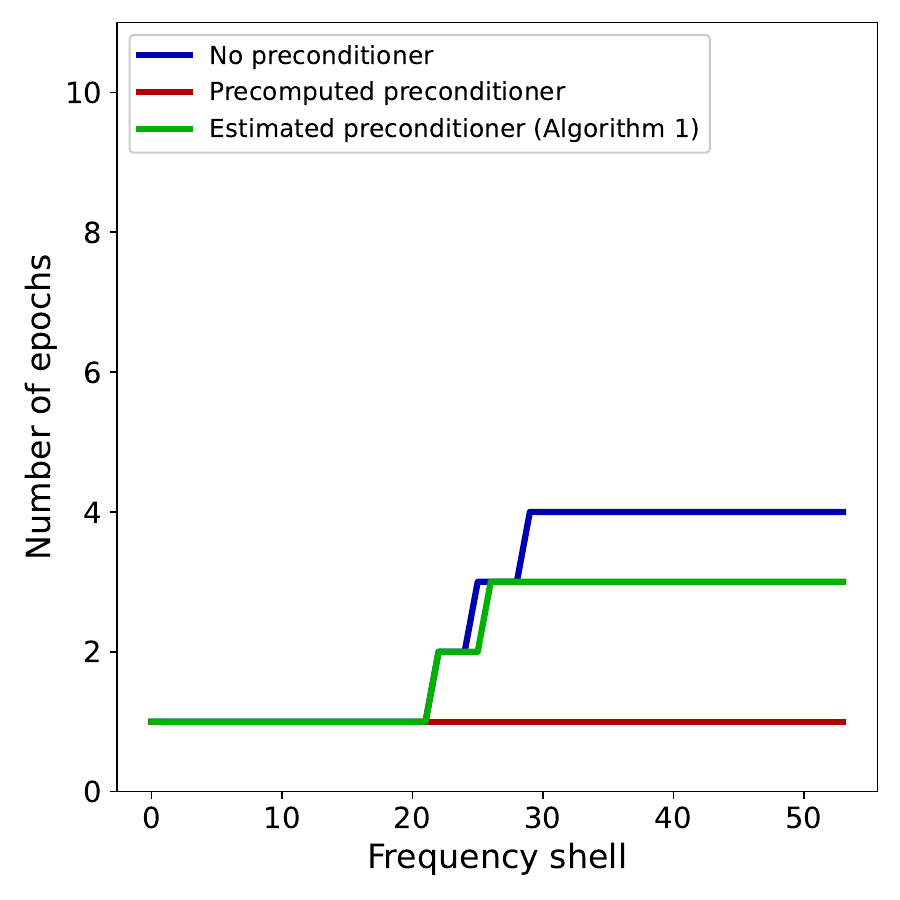}
        \caption{$\text{FSC} \geq 0.75$}
    \end{subfigure}
    \begin{subfigure}[b]{0.32\textwidth}
        \centering
        \includegraphics[width=\textwidth]{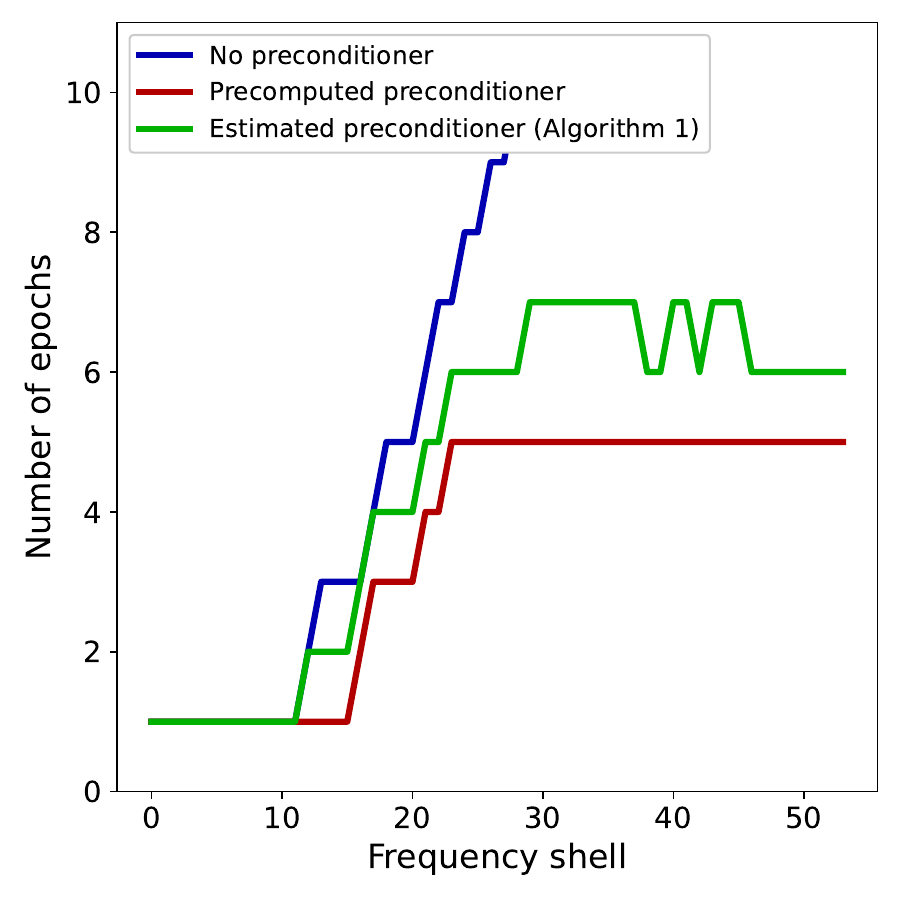}
        \caption{$\text{FSC} \geq 0.9$}
    \end{subfigure}
    \caption{The number of epochs after which the FSC with the ground truth, for 
        individual frequency shells, is greater than $\theta$ for SGD with no preconditioner
        (blue), SGD with a precomputed preconditioner (red) and SGD with an estimated
        preconditioner (green). (a) $\theta = 0.75$ and (b) $\theta=0.9$.}
    \label{fig:num iters real data}
\end{figure}

%\newpage

\section{Outlook and conclusion}
\label{sec:conclusion}

In this article, we analyzed the the conditioning of the cryo-EM reconstruction 
problem with a view towards applying stochastic gradient methods efficiently at 
high resolution. The proposed preconditioner construction and the numerical experiments
performed show promising results for high-resolution 3D refinement. 
While the analysis and experiments presented hold in the special case when the pose 
variables are known, this simplified setting captures the main difficulty of applying 
gradient-based algorithms at high resolution, namely the large condition number 
of the Hessian of the loss function due to the particular structure of the projection 
operator in the Fourier domain.

This proof-of-concept work shows the potential of the SGD algorithm for the more 
general cryo-EM reconstruction problem. There are a number of benefits that such 
an approach would provide:

\begin{enumerate}
    \item The main advantage is the improved convergence speed. While the EM 
        algorithm requires a full pass through the entire dataset at each iteration, 
        SGD methods only use a mini-batch of the particle images. By estimating
        a preconditioner during the reconstruction process, the convergence
        of SGD improves at high resolution, while also benefiting from the speed 
        of using mini-batches. In contrast, EM implicitly computes the same
        preconditioner, but at an increased computational cost due to requiring 
        the entire dataset at each iteration.

    \item A critical component of the current EM approaches is the $\ell_2$ regularizer,
        which makes the maximization step computationally tractable.
        SGD, on the other hand, is compatible with other regularization methods, 
        and one could take advantage, for example, of the Wilson 
        prior~\cite{gilles_molecular_2022} and learned regularization 
        methods~\cite{kimanius_exploiting_2021,kamilov_plug-and-play_2023,kimanius_data-driven_2024}. 
    
    \item While the numerical experiments presented here illustrate the performance
        of the estimated preconditioner with a simple SGD implementation, the 
        preconditioner is compatible with existing and more sophisticated stochastic
        gradient methods used in established cryo-EM software. Unifying the steps 
        of the ab initio reconstruction and high-resolution refinement using a 
        single algorithm is not only more consistent conceptually, but also a 
        practical improvement, allowing a more streamlined implementation.
\end{enumerate}

We defer to future work the analysis and the preconditioner construction in the 
general case, where the pose variables are not known. The main difficulty in the 
general case over the setting of our analysis is that, when marginalizing over 
the unknown poses (see~\eqref{eq:likelihood}), the objective function is no longer 
quadratic. Therefore, the Hessian depends on the current volume iterate. 
However, it is expected that the pose prior distributions are already narrow at 
high resolution and do not vary considerably from one epoch to another. 
Moreover, at each iteration, only a subset of the pose variables are updated (those 
corresponding to the particle images in the current mini-batch), and existing 
efficient pose sampling techniques can be used.
Therefore, an approach to estimate the preconditioner in the general case based on
similar ideas to the ones presented in this article is a promising future direction.

\newpage

\section{Acknowledgments}

The numerical experiments in their final form have been performed using the computational infrastructure at the MRC LMB; therefore, BT thanks Jake Grimmett, Toby Darling and Ivan Clayson of LMB Scientific Computing for computing support.

This work was supported by the NIH under grant R01GM136780, the Air Force Office of Scientific Research under grant FA9550-21-1-0317, the Alfred P. Sloan Foundation under grant FG-2023-20853, DARPA/DOD under grant HR00112490485, and the Simons Foundation. 
This research was performed while BT was affiliated with Yale University, with part of it while visiting the Institute for Pure and Applied Mathematics (IPAM), which is supported by the National Science Foundation (Grant No. DMS-1925919).

\section{Data availability}

The code for reproducing the numerical experiments is available on GitHub at\newline
\url{https://github.com/bogdantoader/simplecryoem} and the particle metadata used for the numerical experiments as well as the output volumes and figures are available on Zenodo at\newline
\url{https://doi.org/10.5281/zenodo.14017756}.

\printbibliography

\end{document}